\documentclass[11pt,reqno]{amsart}
\usepackage{amsthm, amssymb, amsmath, amsfonts, mathtools, mathrsfs}
\usepackage{enumitem}
\usepackage{lineno, tipa, color}
\usepackage{pst-node}
\usepackage{tikz-cd} 
\usepackage{ dsfont }

\allowdisplaybreaks
\newtheorem{theorem}{Theorem}
\newtheorem{lemma}[theorem]{Lemma}

\newtheorem{proposition}[theorem]{Proposition}

\newtheorem{definition}[theorem]{Definition}

\newtheorem{remark}[theorem]{Remark}

\theoremstyle{definition}

\newcommand{\m}{\mathfrak m}
\newcommand{\p}{\mathfrak p}

\def\N{{\mathcal N}}
\def\R{{\mathcal R}}

\def\I{{\mathcal I}}

\DeclareMathOperator{\h}{H}
\DeclareMathOperator{\lh}{h}
\DeclareMathOperator{\depth}{depth}

\DeclareMathOperator{\supp}{Supp}

 \numberwithin{theorem}{section}

\title[Buchsbaumness of the associated graded ring]
{Buchsbaumness of the associated graded rings of filtration}

\date{\today}

 \thanks{The author was supported by NBHM-DAE funded by Govt. of India and Chennai Mathematical Institute}

\keywords{Buchsbaum modules, $\mathds{I}$-invariant, local cohomology, d-sequence, Koszul homology.}

\subjclass[2010]{Primary: 13H10, 13A30, Secondary: 13D45}

\author[Kumari Saloni] {Kumari Saloni}
\address{Department of Mathematics, Indian Institute of Technology Patna, Bihta, Patna 801106, India}
\email{ksaloni@iitp.ac.in, sin.saloni@gmail.com}

\begin{document}
	\maketitle
\sloppy
\begin{abstract}
	Let $(A,\m)$ be a Noetherian local ring of dimension $d>0$ and $I$ an $\m$-primary ideal of $A$. In this paper, we discuss a sufficient condition, 
	for the Buchsbaumness of the local ring $A$ to be passed onto the associated graded ring of filtration. Let $\I$ denote an $I$-good filtration. 
	We prove that if $A$ is Buchsbaum and 
	the $\mathds{I}$-invariant, $\mathds{I}(A)$ and $\mathds{I}(G(\I))$, coincide then the associated graded ring $G(\I)$ is Buchsbaum. As an application
	of our result, we indicate an alternative proof of a conjecture, of Corso on certain boundary conditions for Hilbert coefficients.
\end{abstract}
\section{Introduction}\label{section-intro}
The purpose of this paper is to examine the Buchsbaum property of associated graded rings of filtration. Let $(A,\m)$ be a Noetherian local ring of dimension
$d>0$ with maximal
ideal $\m$ and $M$  a finite $A$-module. We say that $M$ is a {\textit{Buchsbaum module}} if the invariant $\ell_A(M/QM)-e(Q,M)$ is independent of the choice of parameter ideals
$Q$ of $M$, where $\ell_A(M/QM)$ and $e(Q,M)$ denote the length of $M/QM$ and the multiplicity of $M$ relative to $Q$, respectively. 
Note that a Cohen-Macaulay module is Buchsbaum with $\ell_A(M/QM)-e(Q,M)=0$ for 
every parameter ideal $Q$ and the converse also holds. Let $I$ be an $\m$-primary ideal of $A$ and $R(I)=\mathop\oplus\limits_{n\geq 0}I^n$ the Rees algebra of $I$. The associated 
graded ring $G(I)=\mathop\oplus\limits_{n\geq0} I^n/I^{n+1}$ is called 
Buchsbaum if the localization $G(I)_{\N}$ is Buchsbaum as an $R(I)_{\N}$-module where $\N$ is the unique homogeneous maximal ideal of $R(I)$.

The structural properties of the local ring and the associated graded ring are deeply connected to each other. 
Many important properties, such as reducedness, normality, Cohen-Macaulayness etc., are inherited by the local ring $A$ from the associated graded ring 
 (however, this is not true for Buchsbaum rings, see \cite[Example 4.10]{G84} for a counter example).
We are interested  in the inheritance of such properties while passing from $A$ to $G(I)$ which is important  in the theory of blowing-up rings.
In this process, many good properties are lost.   
A great deal of research has been focused on studying the cases where this loss could be avoided. For example, if $A$ is Cohen-Macaulay 
and the reduction number of $I$ is at most one, then $G(I)$ is Cohen-Macaulay. Since the notion of Buchsbaum modules is a generalization of 
Cohen-Macaulay modules, an interesting problem is to study the cases when the associated graded ring 
of a Buchsbaum local ring is Buchsbaum. 

Now suppose that $A$ is a Buchsbaum local ring. Goto proved the following results: (i) if $A$ has maximal embedding dimension, then $G(\m)$ is Buchsbaum
\cite{G82}, (ii) if  $\dim A\geq 2$, $\depth A>0$ and $A$ has multiplicity $2$, then $G(\m)$ is Buchsbaum \cite{goto82mult2} and (iii) 
 $G(Q)$ is Buchsbaum  for  a parameter ideal $Q$ \cite{G83}. 
 Certain boundary conditions on the Hilbert coefficients $e_i(I)$ are also expected to force Buchsbaumness of $G(I)$ which we discuss briefly at the end 
 of the paper, see \cite{corso, GN03}. 
 Furthermore, the $\m$-primary ideals with reduction number one exhibit nice 
 properties. Many authors have investigated the Buchsbaum properties of associated graded rings of such ideals, see \cite{G84, GN03, N98}. 
 Nakamura in \cite{N98} proved the following result which generalizes Goto's result, mentioned above, for parameter ideals.  
If $I^2=QI$ for some minimal reduction $Q=(a_1,\ldots,a_d)$ of $I$, then $G(I)$ is Buchsbaum if and only if the equality 
$(a_1^2,\ldots,a_d^2)\cap I^n=(a_1^2,\ldots,a_d^2)I^{n-2}$ holds for $3\leq n\leq d+1.$ In most of the cases, when $G(I)$ is Buchsbaum, 
the $\mathds{I}$-invariant of $G(I)$ coincides with that of $A$. Recall that the {\textit{$\mathds{I}$-invariant}} of an $A$-module $M$ of dimension $s$ is 
\begin{align}\label{intro-eq1} \mathds{I}(M)= \mathop\sum_{i=0}^{s-1}\binom{s-1}{i}\lh^i(M),\end{align}
where $\lh^i(M)=\ell_A(\h_{\m}^i(M))$ and $\h_{\m}^i(M)$ is the local cohomology module of $M$ with support in the maximal ideal $\m$. 
Indeed, Yamagishi \cite{Yam1,Yam2} removed the condition on the reduction number of $I$ and  proved that  $G(I)$ is Buchsbaum and 
$\lh^i(G(I))=\lh^i(A)$ for $0\leq i< d$ if and only if $\mathds{I}(G(I))=\mathds{I}(A)$. Here, 
$\lh^i(G(I))=\ell(\h_{\N}^i(G(I))$ is the length of the local cohomology module of $G(I)$ with support in $\N$ and 
$\mathds{I}(G(I))=\mathop\sum\limits_{i=0}^{s-1}\binom{d-1}{i}\lh^i(G(I))$. See Section \ref{section-pre} for notations. 

We extend the above theory to the framework of good filtration. An {\textit{$I$-good filtration}} of $A$ is a sequence of ideals $\I=\{I_n\}_{n\geq 0}$ such that $I_{n+1}\subseteq I_{n}$, 
$II_n\subseteq I_{n+1}$ for all $n>0$ and $I_{n+1}=II_n$ for all $n\gg 0$, see \cite{rossi-valla}.
If $A$ is analytically unramified then the filtration of integral closures and the tight closures (provided $A$ contains a field) of powers of $I$ are 
$I$-good filtration. Thus the study of good filtration is in itself interesting. Moreover, many results in literature concerning $I$-adic 
filtration $\{I^n\}_{n\geq 0}$ require non adic-filtration in their proofs such as Ratliff-Rush filtration and the filtration $\{QI^n\}$ with $Q$ a reduction of $I$ etc.
 For instance, see the method of proof for the bound on reduction number by Rossi in \cite{rossi}. 
 Another interesting reason to consider this general set up of $I$-good filtration is its application 
in the study of other graded algebras such as Sally modules, fiber cones and symmetric algebras. 
For instance, Rossi and Valla in \cite{rossi-valla}, have developed very elegant general methods for the above graded algebras using the theory of $I$-good 
filtration. 

 Let $\I=\{I_n\}$ be an $I$-good filtration of $A$. A {\textit{reduction of $\I$}} is an ideal $Q\subseteq I$ such that 
 $I_{n+1}=QI_n$ for $n\gg 0$. Let $R(\I)=\mathop\oplus\limits_{n\geq 0}I_n$ and $G(\I)=\mathop\oplus\limits_{n\geq 0}I_n/I_{n+1}$ denote the 
 Rees algebra and the associated graded ring of $\I$. We write $\N$ for the unique homogeneous maximal ideal $\m R(\I)+R(\I)_+$ of $R(\I)$. As earlier, $\lh^i(A)$ (resp. $\lh^i(G(\I))$) 
denotes the length of the local cohomology modules $\h^i_\m(A)$ of $A$ with support in $\m$ (resp. $\h^i_{\N}(G(\I))$ of $G(\I)$ with support in $\N$).  We say that
$G(\I)$ is Buchsbaum (resp. quasi-Buchsbaum) over $R(\I)$ if 
$G(\I)_{\N}$ is Buchsbaum (resp. quasi-Buchsbaum) over $R(I)_{\N}$. The $\mathds{I}$-invariant of $G(\I)$ can be defined analogous to \eqref{intro-eq1}. 
With this setting, we prove the following results in this paper. 
\begin{theorem}\label{theorem-intro-1}
 Let $A$ be a quasi-Buchsbaum local ring and suppose that the equality $\mathds{I}(G(\I))=\mathds{I}(A)$ holds. Then $G(\I)$ is quasi-Buchsbaum over $R(\I)$.
\end{theorem}
\begin{theorem}\label{theorem-intro-2}
 Let $A$ be a Buchsbaum local ring of dimension $d>0$.
 Let $Q=(a_1,\ldots,a_d)\subseteq I_1$ be a reduction of $\I$ and $r\geq 0$ be an integer such that $I_{n+1}=QI_n$ for all $n\geq r$. 
 Then the following statements are equivalent. 
 \begin{enumerate}
  \item \label{main-theorem-item-1} $G(\I)$ is a Buchsbaum $R(\I)$-module and $\lh^i(G(\I))=\lh^i(A)$ for $0\leq i<d$. 
  \item \label{main-theorem-item-2} The equality $\mathds{I}(G(\I))=\mathds{I}(A)$ holds. 
  \item \label{main-theorem-item-3} The equality $(a_1^{2},\ldots,a_d^{2})\cap I_n=(a_1^{2},\ldots,a_d^{2})I_{n-2}$  holds for all $2< n\leq d+r$.
 \end{enumerate}
\end{theorem}

This paper is organized as follows. Section \ref{section-pre} is devoted to recalling the basic definitions  
and  extend certain facts about the  $\mathds{I}$-invariant and d-sequences to the case of $I$-good filtration. In Section \ref{section-I-equi-case}, 
we consider the equi-$\mathds{I}$-invariant case, i.e., when the equality $\mathds{I}(G(\I))=\mathds{I}(A)$ holds. Theorem \ref{theorem-intro-1} and most of Theorem 
\ref{theorem-intro-2} are proved in this section. The methods of our proofs are inspired by the work of Yamagishi in \cite{Yam1} and \cite{Yam2}. However, 
obtaining the Buchsbaumness of $G(\I)$ from the equality $\mathds{I}(G(\I))=\mathds{I}(A)$ is difficult because the generators of unique homogeneous maximal ideals 
of $R(\I)$ or $G(\I)$ are not necessarily of degree one. Therefore, Yamagishi's methods, which are heavily dependent on the existence of a certain
generating set of $\N$ in degree one,  can not be applied in its present form for filtration. In Section \ref{section-generating-set-of-N}, we 
discuss the existence of a suitable generating set of $\N$ which allows to generalize the methods of \cite{Yam1} to the case of filtration. In Section 
\ref{section-main-proof}, we complete the proof of Theorem \ref{theorem-intro-2} by establishing the Buchsbaumness of $G(\I)$. 
As an application of our results, we indicate an alternative proof of a conjecture, stated by Corso in \cite{corso}, on Hilbert coefficients and Buchsbaumness of $G(I)$.

In this paper, let $M$ denote a finite $A$-module of dimension $s>0$. We write $\ell_A(M)$ for the length of $M$ and $\mu(M)$ for the  minimal number of generators
of $M$. We set $\nu:=\mu(\m)$. 
Let $U=\h^0_\m(A)$. For a quotient ring $A^\prime=A/J$, we write $\I A^\prime$ for the filtration $\{(I_n+J)/J\}_{n\geq 0}$ of the ring $A^\prime$.
It is easy to see that 
$\I A^\prime$ is an $IA^\prime$-good filtration. 
 The multiplicity of $A$ relative to the filtration $\I$ is written as 
 $e(\I)$ instead of $e(\I,A)$. If needed, we write an element $a\in R(\I)_n$ of degree $n$ as $at^n$ since $R(\I)\subseteq R[t]$. 
For a graded module $L$, the $n$-th graded component is denote by $[L]_n$. We assume that the residue field $A/\m$ is infinite.
\section{Preliminaries}\label{section-pre}
In this section, we discuss some results on generalized Cohen-Macaulay modules and their $\mathds{I}$-invariant.
The notion of generalized Cohen-Macaulay modules 
was first discussed in \cite{CST}. 
For a parameter ideal $Q$ of $M$, we define $\mathds{I}(Q;M):=\ell_A(M/QM)-e(Q,M)$. If $M$ is Cohen-Macaulay, $\mathds{I}(Q;M)=0$ for any parameter ideal $Q$.
The ideas discussed in \cite{CST} and \cite{Sc}, led to the study of modules for which  $\mathds{I}(Q,M))$, a measure of non Cohen-Macaulayness, is independent of 
the parameter ideals $Q$ of $M$. Such modules have many properties similar to those of Cohen-Macaulay modules. We refer to \cite{vogel} for
details. The $\mathds{I}$-invariant, also known as the Buchsbaum-invariant, of $M$ is defined as 
\[\mathds{I}(M):=\sup \{\mathds{I}(Q;M):Q \mbox{ is a parameter ideal of }M.\}\]  
In general, $0\leq \mathds{I}(M)\leq \infty$ and $\mathds{I}(M)$ is finite if and only if 
$M$ is generalized Cohen-Macaulay. Equivalently, $M$ is generalized Cohen-Macaulay if $\ell_A(\h_\m^i(M))<\infty$ for $0\leq i<s$.
In this case, 
 \[\mathds{I}(M)=\sum_{i=0}^{s-1}\binom{s-1}{i} \ell_A(\h^i_\m(M)).\]
 A system of parameters $a_1,\ldots,a_s\in\m$ of $M$ is called {\it standard} if $\mathds{I}(M)=\mathds{I}(Q;M)$ where $Q=(a_1,\ldots,a_s)$.
 An ideal $I$ with $\ell(M/IM)<\infty$ is called a {\it standard ideal for $M$} if every system of parameters of $M$ contained in $I$ is standard.
 Every system of parameters of $M$ is standard, i.e., $\m$ is a standard ideal for $M$, if and only if $M$ is  Buchsbaum. 
 We say that $M$ is {\it quasi-Buchsbaum} if $\m^2$ is a standard ideal for $M$. 
Clearly, Buchsbaum modules are quasi-Buchsbaum and $\mathds{I}(M)<\infty$ whenever
 $M$ is quasi-Buchsbaum. 
 We say that $G(\I)$ is quasi-Buchsbaum (respectively generalized Cohen-Macaulay) if $G(\I)_\N$ is
 quasi-Buchsbaum (respectively generalized Cohen-Macaulay) $R(\I)_\N$-module. 
 Note that $\h^i_{\N R(\I)_\N}(G(\I)_\N)\cong \h^i_\N(G(\I))$ for each $i$. 
 We have that
 \[ \mathds{I}(G(\I))=\mathds{I}(G(\I)_\N)=\sum_{i=0}^{s-1}\binom{s-1}{i} \ell_{R(\I)}(\h^i_\N(G(\I))).\] 
For $I$-adic filtration, i.e., $I_n=I^n$ for all $n$, it is known that $\mathds{I}(G(I))\geq \mathds{I}(A)$.
 The following result is a generalization of \cite[Lemma 5.1]{trung-gcm}.
\begin{proposition}\label{proposition-lemma-3.3}\label{lemma-5.1/trung}
Let $p\geq 1$ be an integer and let $a_1,\ldots,a_d\subseteq I_p$ a sequence of elements such that 
 $a_1t^p,\ldots,a_dt^p\in R(\I)_p$ is a system of parameters of $G(\I)$.  Then
 \[\mathds{I}((a_1t^p,\ldots,a_dt^p); G(\I))\geq \mathds{I}((a_1,\ldots,a_d); A)\] and
 equality holds if and only if $(a_1,\ldots,a_d)\cap I_k=(a_1,\ldots,a_d)I_{k-p}$ for all $k\geq 0$.
\end{proposition}
\begin{proof}
For $n\geq 0$, we have $\ell(A/(a_1,\ldots,a_d)^n)\geq \ell(A/ I_{np})$. Therefore 
\begin{equation}\label{eqn-for-e0} e((a_1,\ldots,a_d),A)\geq e(\I)p^d.\end{equation}
Let $k_0\geq 0$ be an integer such that $I_{k+1}=(a_1,\ldots,a_d)I_{k+1-p}$ for all $k\geq k_0$. Then for $n\geq 0$ and $k\geq k_0+np$, 
$I_{k+1}=(a_1,\ldots,a_d)^n I_{k+1-np}.$ Put $k(n)=k_0+np+1$, Then, for all $n\geq 0$, 
\begin{align}
\ell(A/I_{k(n)})=\sum_{k=0}^{k(n)-1}\ell(I_k/I_{k+1})&\geq\ell(G(\I)/(a_1t^p,\ldots,a_dt^p)^n G(\I))\nonumber\\
&=\sum_{k=0}^\infty \ell(I_k/((a_1,\ldots,a_d)^n I_{k-np}+I_{k+1}))\nonumber\\ 
&=\sum_{k=0}^{k(n)-1} \ell(I_k/((a_1,\ldots,a_d)^n I_{k-np}+ I_{k+1}))\nonumber\\
 &\geq\sum_{k=0}^{k(n)-1} \ell\big(I_k/((a_1,\ldots,a_d)^n\cap I_k)+I_{k+1})\big)\label{eqn-inside-proposition-d-seq-0}\\
 & \geq\sum_{k=0}^{k(n)-1} \ell((I_k+(a_1,\ldots,a_d)^n)/(I_{k+1}+(a_1,\ldots,a_d)^n))\nonumber\\
 &=\ell(A/((a_1,\ldots,a_d)^n+I_{k(n)}))=\ell(A/(a_1,\ldots,a_d)^n).\label{eqn-inside-proposition-d-seq}
\end{align}
Thus $e(\I)p^d\geq e((a_1t^p,\ldots,a_dt^p),G(\I))\geq e((a_1,\ldots,a_d),A)$. Now using \eqref{eqn-for-e0}, we get that 
\begin{equation}\label{eqn-for-e0-2}
 e((a_1t^p,\ldots,a_dt^p),G(\I))=e((a_1,\ldots,a_d),A).
\end{equation}
Further, putting $n=1$ in \eqref{eqn-inside-proposition-d-seq}, we get 
\begin{equation*}
 \ell(G(\I)/(a_1t^p,\ldots,a_dt^p)G(\I))- e((a_1t^p,\ldots,a_dt^p),G(\I))\geq \ell(A/(a_1,\ldots,a_d))-e((a_1,\ldots,a_d),A).
\end{equation*}
Thus \[\mathds{I}((a_1t^p,\ldots,a_dt^p); G(\I))\geq \mathds{I}((a_1,\ldots,a_d);A).\]

Suppose equality holds. Then $\ell(G(\I)/(a_1t^p,\ldots,a_dt^p)G(\I))=\ell(A/(a_1,\ldots,a_d))$. 
  For $k>k_0$, $(a_1,\ldots,a_d)\cap I_k=(a_1,\ldots,a_d)\cap (a_1,\ldots,a_d)I_{k-p}= (a_1,\ldots,a_d)I_{k-p}.$
Suppose $k\leq k_0$. Then we have 
$(a_1,\ldots,a_d)\cap I_k \subseteq (a_1,\ldots,a_d)I_{k-p}+I_{k+1}$ from the equality in \eqref{eqn-inside-proposition-d-seq-0}. Therefore, 
\begin{align*}
 (a_1,\ldots,a_d)\cap I_k &\subseteq (a_1,\ldots,a_d) I_{k-p}+((a_1,\ldots,a_d)\cap I_{k+1})\\
 &\subseteq  (a_1,\ldots,a_d)I_{k-p}+((a_1,\ldots,a_d)I_{k+1-p}+((a_1,\ldots,a_d)\cap I_{k+2})\\
 &\vdots\\
 &\subseteq (a_1,\ldots,a_d)I_{k-p}+(a_1,\ldots,a_d)I_{k_0-p}\\
&=(a_1,\ldots,a_d)I_{k-p}.
\end{align*}
Clearly, if the above equality holds for all $k>0$, then $\ell(G(\I)/(a_1t^p,\ldots,a_dt^p)G(\I))=\ell(A/(a_1,\ldots,a_d)).$
\end{proof}
 The role played by regular sequences in the study of Cohen-Macaulay modules is broadly replaced
 by d-sequences in case of generalized Cohen-Macaulay modules. 
A sequence $a_1,\ldots,a_s$ of elements of $A$ is said to be a {\textit{d-sequence}} on $M$ if the equality 
 \[q_{i-1}M:a_ia_j=q_{i-1}:a_j\] holds for $1\leq i\leq j\leq s$ where 
 $q_{i-1}=(a_1,\ldots,a_{i-1})$ and $q_0=(0)$. It is said to be an {\textit{unconditioned d-sequence}} if it is a d-sequence in any order. Moreover, we 
 say that $a_1,\ldots,a_s$ is an {\textit{unconditioned strong d-sequence}} (u.s.d-sequence) on $M$ if $a_1^{n_1},\ldots,a_s^{n_s}$ is an 
 unconditioned d-sequence on $M$ for  all integers $n_1,\ldots,n_s>0$. 
 Recall that a standard system of parameters is a d-sequence and a module is Buchsbaum if and only if every system of parameters 
 is u-s-d-sequence.  Indeed, we have the following characterization for standard system of parameters. 
 \begin{remark}\label{remark-from-trung-huneke}\cite[Theorem A]{Sc}\cite[Theorem 2.1]{trung-gcm} 
 For any system of parameters $a_1,\ldots,a_s$ of $M$, the following conditions are equivalent:
 \begin{enumerate}
  \item $a_1,\ldots,a_s$ is an u.s.d-sequence on $M$;
  \item $a_1,\ldots,a_s$ is a standard system of parameters of $M$, i.e., $\mathds{I}((a_1,\ldots,a_s);M)=\mathds{I}(M)$;
  \item the equality $\mathds{I}((a_1,\ldots,a_s);M)=\mathds{I}((a_1^2,\ldots,a_s^2);M)$ holds.
 \end{enumerate}
\end{remark}
The following proposition is discussed in \cite[Proposition 2.2]{Yam2} except the last part. 
\begin{proposition}\label{proposition-2.2-Y}
 Let $a_1,\ldots,a_d$ be a sequence of elements in $I_1$.
 Let $m>0$ be an integer such that $a_1^{n_1},\ldots,a_d^{n_d}$ is an unconditioned d-sequence for all $n_i\geq m$ 
 and $1\leq i\leq d$. Then the following conditions are equivalent:
 \begin{enumerate}
  \item\label{key1} The equality 
  \[ (a_i^{n_i}~|~i\in \Lambda)\cap I_n=\sum_{i\in \Lambda} a_i^{n_i}I_{n-n_i} \]
  holds for all $\Lambda\subseteq [1,d]$, $n\in\mathbb{Z}$ and $n_i\geq m$;
  \item\label{key2} The equality 
 \[ (a_1^{2m},\ldots,a_d^{2m})\cap I_n= (a_1^{2m},\ldots,a_d^{2m})I_{n-2m}  \]
 holds for all $n\in\mathbb{Z}$.
 \end{enumerate}
When this is the case,  \[G(\I)/(a_it)^{n_i}G(\I)\cong G(\I A/(a_i^{n_i}))\]
as graded $R(\I)$-modules for all $n_i\geq m$ and $1\leq i\leq d$.
\end{proposition}
\begin{proof}
Equivalence of \eqref{key1} and \eqref{key2} follows from \cite[Proposition 2.2]{Yam2}.
We only prove the last part. We have that, for all $k\geq 0,$
\[ [G(\I)/(a_it)^{n_i}G(\I)]_k \cong \frac{I_k}{a_i^{n_i}I_{k-n_i}+I_{k+1}}
\cong\frac{I_k}{(a_i^{n_i}\cap I_{k})+I_{k+1}}
\cong \frac{I_k+(a_i^{n_i})}{I_{k+1}+(a_i^{n_i})}
\cong [G(\I A/(a_i^{n_i}))]_k.\]
\end{proof}
It is often useful to consider rings of positive depth. Particularly, it provides a way to reduce a problem to lower dimensional cases and apply 
method of induction. We now briefly relate the properties of $\mathds{I}$-invariant in the ring $A$ and $A/U$ where $U=\h_\m^0(A)$.
Consider the exact sequence 
\[0\to U^{*}\to G(\I)\xrightarrow{\phi}G(\I A/U)\to 0\]
of graded $R(\I)$-modules where $\phi:G(\I)\to G(\I A/U)$ is the canonical epimorphism of graded modules induced from the projection map $A\to A/U$ and  $U^*=\ker\phi.$ Since $[G(\I A/U)]_n=(I_n+U)/(I_{n+1}+U)\cong I_n/((U\cap I_n)+I_{n+1})$, we have 
$$[U^{*}]_n=((U\cap I_n)+I_{n+1})/I_{n+1}\cong (U\cap I_n)/(U\cap I_{n+1})$$
for $n\in\mathbb{Z}$. Thus $\ell_A(U^*)=\ell_A(U)=\ell_A(\h^0_\m(A))$. 
Then, applying the local
cohomology functor to the above sequence, we get the short exact sequence 
\[ 0\to U^*\to \h^0_{\mathcal{N}}(G(\I))\to \h^0_{\mathcal{N}}(G(\I A/U))\to 0\]  and the isomorphism
\[ \h^i_{\mathcal{N}}(G(\I))\cong \h^i_{\mathcal{N}}(G(\I A/U)) \text{ for all } i\geq 1.\]
\begin{lemma}\label{Lemma-3.5-Y} With the notations as above 
\begin{enumerate}
 \item \label{key3} $\lh^0(G(\I A/U))=\lh^0(G(\I))-\ell(U^*)=\lh^0(G(\I))-\lh^0(A)$  and $\lh^i(G(\I /U))=\lh^i(G(\I))$ for $i\geq 1$.
\item \label{key4} The following conditions are equivalent:
\begin{enumerate}
 \item $\lh^0(G(\I A/U))=0$;
 \item $\lh^0(G(\I))=\lh^0(A)$;
 \item $\h^0_\N(G(\I))=U^*$. 
\end{enumerate}
\item\label{key5} Suppose $A$ is generalized Cohen-Macaulay and $\mathds{I}(G(\I))=\mathds{I}(A)$. Then $\mathds{I}(G(\I A/U))=\mathds{I}(A/U)$.
\end{enumerate}
\end{lemma}
\begin{proof}
The assertions \eqref{key3} and \eqref{key4} are  discussed above.
 For \eqref{key5}, suppose $\mathds{I}(G(\I))=\mathds{I}(A)$. Then, by assertion \eqref{key1}, \[\mathds{I}(G(\I A/U))=\mathds{I}(G(\I))-\lh^0(A)=\mathds{I}(A)-\lh^0(A)=\mathds{I}(A/U)\]
\end{proof}
\section{Equi-$\mathds{I}$-invariant case and quasi-Buchsbaumness of $G(\I)$}\label{section-I-equi-case}
Let $(a_1,\ldots,a_d)\subseteq I_1$ be a reduction of $\I$ and $r\geq 0$ an integer such that
$I_{n+1}=(a_1,\ldots,a_d)I_n$ for all $n\geq r$. We write $Q$ for the ideal $(a_1,\ldots,a_d)$. For this section, we assume that 
$A$ is generalized Cohen-Macaulay. We prove several equivalent conditions for the equi-$\mathds{I}$-invariant case, i.e., when the equality 
$\mathds{I}(G(\I))=\mathds{I}(A)$ holds. The results of this 
section are discussed in \cite{Yam2} for $I$-adic filtration. Similar methods extend the results to filtration.  

\begin{proposition}\label{prop-I-equi-1}
Let  $m>0$ be an integer such that $a_1^{n_1},\ldots,a_d^{n_d}$ is an unconditioned d-sequence for
all $n_i\geq m$, $1\leq i\leq d$. Then the following conditions are equivalent:
\begin{enumerate}
 \item \label{prop-I-equi-1-item-1} $\mathds{I}(G(\I))=\mathds{I}(A)$;
 \item \label{prop-I-equi-1-item-2} The equality $(a_1^{2m},\ldots,a_d^{2m})\cap I_n=(a_1^{2m},\ldots,a_d^{2m})I_{n-2m}$ holds for $2m<n\leq d(2m-1)+r$;
 \item \label{prop-I-equi-1-item-3} The equality $(a_1^{2m},\ldots,a_d^{2m})\cap I_n=(a_1^{2m},\ldots,a_d^{2m})I_{n-2m}$  holds for all $n\in \mathbb{Z}$.
\end{enumerate}
When this is the case, $(a_1t)^{n_1},(a_2t)^{n_2},\ldots,(a_dt)^{n_d}$ is an unconditioned d-sequence for $G(\I)$ for all 
$n_i\geq m$, $1\leq i\leq d$. 
\end{proposition}
\begin{proof} 
 Note that $a_1^m,\ldots,a_d^m$ forms an u-s-d-sequence. By Proposition \ref{proposition-lemma-3.3}, we have that,
 for all $l\geq 1$	
 \[\mathds{I}(G(\I))\geq  \mathds{I}(((a_1t)^{lm},\ldots,(a_dt)^{lm});G(\I))\geq \mathds{I}((a_1^{lm},\ldots,a_d^{lm});A)=\mathds{I}(A).\]
If $\mathds{I}(G(\I))=\mathds{I}(A)$, then Proposition \ref{proposition-lemma-3.3} gives that
$(a_1^{lm},\ldots,a_d^{lm})\cap I_n=(a_1^{lm},\ldots,a_d^{lm})I_{n-lm}$ for $n\in \mathbb{Z}$ and for any $l\geq 1$  
which gives $\eqref{prop-I-equi-1-item-3}$. 

Now suppose $\eqref{prop-I-equi-1-item-3}$ holds. By Proposition \ref{proposition-2.2-Y}, 
$(a_1^{lm},\ldots,a_d^{lm})\cap I_n=(a_1^{lm},\ldots,a_d^{lm})I_{n-lm}$
holds for $n\in\mathbb{Z}$ and for all $l\geq 1$. Consequently, Proposition \ref{proposition-lemma-3.3} gives
the equality $\mathds{I}(((a_1t)^{lm},\ldots,(a_dt)^{lm});G(\I))= \mathds{I}((a_1^{lm},\ldots,a_d^{lm});A)=\mathds{I}(A)$
for all $l\geq 1$. Thus $\mathds{I}(G(\I))= \mathds{I}(((a_1t)^{lm},\ldots,(a_dt)^{lm});G(\I))=\mathds{I}(A)$ for all $l\geq 1$ because $(a_1t)^{m},\ldots,(a_dt)^{m})$ forms a u.s.d-sequence and 
by Remark \ref{remark-from-trung-huneke}. This establishes $\eqref{prop-I-equi-1-item-1}\Leftrightarrow \eqref{prop-I-equi-1-item-3}$.

The statements  $\eqref{prop-I-equi-1-item-2}$ and $\eqref{prop-I-equi-1-item-3}$ are equivalent since
 for all $n>d(2m-1)+r$, we have that
 \[ I_n=Q^{n-r}I_r\subseteq (a_1^{2m},\ldots,a_d^{2m})Q^{n-r-2m}I_r \subseteq (a_1^{2m},\ldots,a_d^{2m})I_{n-2m}.\]
The last statement follows from Remark \ref{remark-from-trung-huneke} and the fact that for $n_i\geq m$,
\[\mathds{I}(G(\I))\geq \mathds{I}(((a_1t)^{n_1},\ldots,(a_dt)^{n_d});G(\I))\geq \mathds{I}(((a_1t)^{m},\ldots,(a_dt)^{m});G(\I))=\mathds{I}(G(\I)).\]
 \end{proof}
 We now prove the quasi-Buchsbaumness of the associated graded ring.
 Recall that an $A$-module $M$ is quasi-Buchsbaum if and only at least one (equivalently every) system of parameters contained in $\m^2$ is a weak $M$-sequence. Recall, from \cite{trung-abs}, that a sequence $a_1,\ldots,a_s$ of elements of $A$ is said to be a {\textit{weak $M$-sequence}} if the equality 
 $q_{i-1}M:a_i=q_{i-1}M:\m$ holds for $1\leq i\leq s$ where 
 $q_{i-1}=(a_1,\ldots,a_{i-1})$ and $q_0=(0)$.
 The following result is a generalization of \cite[Theorem 1.2]{Yam2}. See \cite{Yam2} for the converse statement in $I$-adic case.
 \begin{theorem}\label{theorem-quasi-buchs}
Suppose that $A$ is a quasi-Buchsbaum local ring and $\mathds{I}(G(\I))=\mathds{I}(A)$. Then $G(\I)$ is quasi-Buchsbaum over $R(\I)$. 
\end{theorem}
\begin{proof}
 Suppose $A$ is quasi-Buchsbaum and $\mathds{I}(G(\I))=\mathds{I}(A)$. Then $a_1^{n_1},\ldots,a_d^{n_d}$ is a weak sequence and an unconditioned d-sequence in $A$
 for all $n_i\geq 2$. 
 Let $f_i=a_it\in R(\I)_1 $. We show that $f_1^2,\ldots,f_d^2$ is a weak $G(I)_{\N}$-sequence, i.e., 
 \[(f_1^2,\ldots,f_{i-1}^2)G(\I):f_i^2\subseteq (f_1^2,\ldots,f_{i-1}^2)G(\I):\N\]
 for every $1\leq i\leq d.$ For this, we fix an $i$ and let $g$ be a homogeneous element of $G(\I)$ with $\deg g=m$ such that 
 $gf_i^2\subseteq  (f_1^2,\ldots,f_{i-1}^2)G(\I)$. Let $g=x+I_{m+1}$ for some $x\in I_m$. Then there exist an element 
 $y\in (a_1^2,\ldots,a_{i-1}^2)I_m$ and $z\in I_{m+3}$ such that $a_i^2x=y+z$. Since $a_1^{n_1},\ldots,a_d^{n_d}$ is an unconditioned d-sequence in $A$ for all $n_i\geq 2$ and $\mathds{I}(G(\I))=\mathds{I}(A)$, we have 
 $z\in (a_1^2,\ldots,a_{i}^2)\cap I_{m+3}=(a_1^2,\ldots,a_{i}^2)I_{m+1}$ by 
 Proposition \ref{prop-I-equi-1} and Proposition \ref{proposition-2.2-Y}. Let $v\in I_{m+1}$ such that $a_i^2(x-v)\in(a_1^2,\ldots,a_{i-1}^2).$
 
 Now $x-v\in (a_1^2,\ldots,a_{i-1}^2):a_i^2=(a_1^2,\ldots,a_{i-1}^2):\m$ as $a_1^2,\ldots,a_d^2$ is a weak $A$-sequence. Now
 we have that 
 $\m (x-v)\subseteq (a_1^2,\ldots,a_{i-1}^2)\cap I_m\subseteq (a_1^2,\ldots,a_{i-1}^2)I_{m-2}$ and 
 $I_n(x-v)\subseteq (a_1^2,\ldots,a_{i-1}^2)\cap I_{m+n}\subseteq 
 (a_1^2,\ldots,a_{i-1}^2)I_{m+n-2}$ for all $n\geq 1$ by   Proposition \ref{prop-I-equi-1} and Proposition \ref{proposition-2.2-Y} again. 
 This implies that $\N g\subseteq (f_1^2,\ldots,f_{i-1}^2)G(\I)$. Thus $G(\I)$ is a quasi-Buchsbaum module over $R(\I)$. 
\end{proof}
\begin{proposition}\label{proposition-3.6-Y}
 The following conditions are equivalent:
 \begin{enumerate}
  \item \label{item-1-prop-3.6-Y} $\mathds{I}(G(\I))=\mathds{I}(A)$;
  \item \label{item-2-prop-3.6-Y} $\lh^i(G(\I))=\lh^i(A)$ for $0\leq i<d.$
 \end{enumerate}
 When this is the case, $A/U$ satisfies the above equivalent conditions. Furthermore, we have that
 \begin{enumerate}[label=(\roman{*})]
 \item the sequence 
 $$0\to \h^0_\N(G(\I))\to G(\I)\to G(\I A/U)\to 0$$ 
of graded $R(\I)$-modules is exact, $[\h^0_\N(G(\I))]_n\cong (U\cap I_n)/(U\cap I_{n+1})$ for all $n\in\mathbb{Z}$ and
  \item if $\m \h^0_\m(A)=0,$ then $\N\h^0_\N(G(\I))=0.$
 \end{enumerate}
\end{proposition}
 \begin{proof}
 The last part is clear from Lemma \ref{Lemma-3.5-Y}.

\eqref{item-1-prop-3.6-Y} $\Rightarrow$ \eqref{item-2-prop-3.6-Y} We apply induction on $d$. For $d=1$, it is clear. Suppose $d\geq 2$. 
By passing to $A/U$ and using Lemma \ref{Lemma-3.5-Y}, we may assume that $\depth A>0$. 
Then $\mathds{I}(G(\I))<\infty$, so $G(\I)$ is generalized Cohen-Macaulay. 
Let $a_1t,\ldots,a_dt\in R(\I)_1$ be a system of parameters of $G(\I)$ (see Lemma \ref{lemma-4.1-Y} for existence of such $a_1,\ldots,a_d\in I_1$). 
Let $m>0$ be an integer such that $a_1^{n_1},a_2^{n_2},\ldots,a_d^{n_d}$ forms a standard system of parameters of $A$ 
for each $n_i\geq m$. By Proposition \ref{prop-I-equi-1} and Proposition \ref{proposition-2.2-Y}, $(a_1^m)\cap I_k=a_1^m I_{k-m}$ for all $k\in\mathbb{Z}$
and 
\[G(\I)/(a_1t)^{m}G(\I)\cong G(\I A/(a_1^{m}))\]
as graded $R(\I)$-modules. Since $\depth A>0$, $(a_1t)^m$ is a non-zero-divisor on $G(\I)$. Therefore 
$\lh^0(G(\I))=0.$

Now, consider the short exact sequences 
\[ 0\to A\xrightarrow{a_1^m}A\to A/(a_1^m)\to 0 \mbox{ and }\]
\[0\to G(\I)(-m)\xrightarrow{(a_1t)^m} G(\I)\rightarrow G(\I A/(a_1^m))\to 0.\]
By Proposition \ref{prop-I-equi-1}, $(a_1t)^m,\ldots,(a_dt)^m$ is an u.s.d-sequence on $G(\I)$. 
We can choose $m$ large enough so that 
 $a_1^m$ (resp. $(a_1t)^m$) annihilates the local 
cohomology modules $\h^i_\m(A)$ (resp. $\h^i_\N(G(\I))$) for $0\leq i\leq s-1$. Then, for $0\leq i<s-1$,
the following sequences are exact:
$$0\to \h^i_\m(A)\to\h^i_\m(A/(a_1^m))\to\h^{i+1}_\m(A)\to 0,$$
$$0\to \h^i_\N(G(\I))\to\h^i_\N(G(\I A/(a_1^m)))\to\h^{i+1}_\N(G(\I))(-m)\to 0.$$
Consequently, for $0\leq i<s-1$, we have
\begin{equation}
 \lh^i(A/(a_1^m))=\lh^i(A)+\lh^{i+1}(A) \text{ and}
 \lh^i(G(\I A/(a_1^m)))=\lh^i(G(\I))+\lh^{i+1}(G(\I))\label{eq-2-proposition-3.6-Y}
\end{equation}
 which gives that
$\mathds{I}(G(\I/(a_1^m)))=\mathds{I}(A/(a_1^m))$. By induction hypothesis, $\lh^i(G(\I A/(a_1^m)))=\lh^i( A/(a_1^m))$ for $0\leq i<s-1$. 
Now using  \eqref{eq-2-proposition-3.6-Y} inductively, we get that 
$\lh^i(G(\I))=\lh^i(A)$ for $0\leq i<s.$

 \eqref{item-2-prop-3.6-Y} $\Rightarrow$ \eqref{item-1-prop-3.6-Y} It follows from the definition of $\mathds{I}$-invariant. 
\end{proof}
\section{A generating set of $\N$}\label{section-generating-set-of-N}
In this section, we discuss the existence of a generating set of 
$\N$ consisting of homogeneous elements, not necessarily in the same degree, which possesses the properties of a 
$G(\I)$-basis of $\N$. 
We first recall the following definition from \cite{vogel}. 
\begin{definition}\cite[Definition 1.7]{vogel}
 Let $J$ be an ideal such that $\dim M/JM=0.$ A system of elements $a_1,\ldots,a_t$ of $A$ is called an $M$-basis of $J$ 
 if $(i)$ $a_1,\ldots,a_t$ is a minimal
 generating set of $J$ and $(ii)$ for every system $i_1,\ldots,i_s$ of integers such that $1\leq i_1<i_2<\ldots<i_s\leq t$, the 
 elements $a_{i_1},\ldots,a_{i_s}$  form a system of parameters of $M$.
\end{definition}
Let $R=\mathop\oplus\limits_{n\geq 0} R_n$ be a Noetherian graded ring and 
$N$ a Noetherian graded $R$-module. For a homogeneous ideal $\mathfrak{J}$ of $R$ with $\dim N/\mathfrak{J}N=0$, an {\textit{$N$-basis of $\mathfrak{J}$}} 
consisting of homogeneous elements can be defined in a similar way. In local case, such a basis always exists, see  \cite[Proposition 1.9, Chap I]{vogel}. 
For $k$-algebras, the existence of $N$-bases of homogeneous ideals is discussed in \cite[Section 3, Chap I]{vogel}. The main goal of this section is 
to prove Lemma \ref{lemma-4.1-Y} which gives the existence of a $G(\I)$-basis of $\N$ such that its degree zero elements form an $A$-basis of $\m$. 
First, we briefly discuss preliminary facts. In the rest of this section, 
we assume that $R=\mathop\oplus\limits_{n\geq 0} R_n$ is a Noetherian graded ring where $R_0$ is a local ring with maximal ideal 
$\mathfrak{m}_0$ and infinite residue field $R_0/\mathfrak{m}_0$. Let $\mathcal{M}=\mathfrak{m}_0\oplus\mathop\oplus\limits_{n\geq 1} R_n$. 
\begin{lemma}\label{remark-d-to-r} Let $m_1,\ldots,m_p\in M$ be a generating set of $M$ and $\{M_i~:~1 \leq i\leq k\}$ a finite collection
 of proper submodules of $M$. Then there exists an element 
 $y=\alpha_1m_1+\ldots+\alpha_pm_p\in M\setminus (M_1\cup\ldots\cup M_k)$ where
 $\alpha_j$ is either zero or a unit in $A$ for $1\leq j\leq p$. 
 \end{lemma}
 \begin{proof}
 Since $\frac{M_i+\m M}{\m M}$ is a proper subspace of $\frac{M}{\m M}$ for each $i$, we have that 
 $$\frac{M_1+\m M}{ \m M} \cup \ldots \cup \frac{M_k+\m M}{\m M} \varsubsetneq \frac{M}{\m M}.$$
 \end{proof}
\begin{lemma}\label{lemma-existence-1}
Let $a_1,\ldots,a_p\in R_l$ be homogeneous elements of degree $l>0$ and 
$\mathfrak{J}_1,\ldots,\mathfrak{J}_k$ homogeneous ideals of $R$ such that $(a_1,\ldots,a_p)R\nsubseteq \mathfrak{J}_i$ for all 
$1\leq i\leq k$. Then there exists an element 
$$\alpha_1 a_1+\ldots+\ldots \alpha_pa_p\notin  \mathfrak{J}_1\cup\ldots\cup \mathfrak{J}_k $$ 
where $\alpha_j$ is either zero or a unit in $R_0$ for $1\leq j\leq p$.
\end{lemma}
\begin{proof}
Let $M$ be the $R_0$-submodule of $R_l$ generated by $a_1,\ldots,a_p$ and $M_i=M\cap \mathfrak{J}_i$ for $i=1,\ldots,k$.
Since $(a_1,\ldots,a_p)R\nsubseteq \mathfrak{J}_i$, there exists $a_{i_j}$ for each $\mathfrak{J}_i$ such that $a_{i_j}\notin \mathfrak{J}_i$.
So, $M_i$ is a proper submodule of $M$. Now the conclusion follows from Lemma \ref{remark-d-to-r}.
\end{proof}
For a Noetherian graded $R$-module $N$, we now prove the existence of an $N$-basis of a homogeneous ideal of $R$ generated in degree $g$. Our result is a generalization of 
 \cite[Proposition 1.9, Chap I]{vogel}.
\begin{proposition}\label{proposition-basis-existence-1}
 Let $\mathfrak{J}\subseteq R$ be an ideal generated by homogeneous elements
of degree $g$. Let $N_1,\ldots, N_k$ be Noetherian graded $R$-modules with $\dim N_i/\mathfrak{J}N_i=0$ for $1\leq i\leq k.$ Then there 
exists a system of homogeneous elements of degree $g$ forming $N_i$-basis of $\mathfrak{J}$ for all
$1\leq i\leq k.$
\end{proposition}
\begin{proof}
 Let $\mathfrak{J}=(b_1,\ldots,b_t)R$ with $t=\mu(\mathfrak{J}_\mathcal{M})$ and $\deg b_i=g$ for $1\leq i\leq t.$ We may assume that $\dim N_i=s_i>0$ 
  and $t\geq s_i$ for each $i$. 

{\it{Claim 1.}} For each $0\leq m\leq t$, there exist homogeneous elements $a_1,\ldots,a_m\in \mathfrak{J}$ such that either $a_i=0$ or 
$\deg a_i=g$ for $1\leq i\leq m$ and the following conditions hold:
\begin{enumerate}
 \item \label{key-prop-1} $\mathfrak{J}=(a_1,\ldots,a_m,b_{m+1},\ldots,b_t)R$ and
 \item \label{key-prop-2} for all $l=1,\ldots,k$, for all $j=0,\ldots,\min(s_l,m)$ and all $1\leq i_1< \ldots < i_j\leq m$, $a_{i_1},\ldots,a_{i_j}$
is a part of a system of parameters of $N_l$.
\end{enumerate}
\begin{proof}[Proof of Claim 1] We prove by induction on $m$. For $m=0$, this is trivial. Now let $0<m\leq t$ and there exist elements
$a_1,\ldots,a_{m-1}$ as stated in  Claim 1 such that \eqref{key-prop-1} and \eqref{key-prop-2} are satisfied.
Define
\begin{align*} L=\{&\p\in Spec(R):\p \text{ is homogeneous and there are } l,j,i_1,\ldots,i_j, 1\leq l\leq k, 0\leq j\leq \min(s_l,m),\\
&1\leq i_1<\ldots<i_j<m \text { with } \p\in \supp N_l/(a_{i_1},\ldots,a_{i_j})N_l) \text{ and } \dim R/\p=s_l-j \}.\end{align*}
Clearly,  $\mathfrak{J}\nsubseteq \p$ for all $\p\in L$ and $\mathfrak{J} \nsubseteq (a_1,\ldots,a_{m-1})R$. 
By Lemma \ref{lemma-existence-1}, there exists an element, say $a_m$, such that
\[a_m=\alpha_1 a_1+\ldots+\alpha_{m-1}a_{m-1}+\alpha_{m}b_m+\ldots\alpha_tb_t\notin (a_1,\ldots,a_{m-1})R\cup\mathop\cup\limits_{\p\in L} \p\] with 
$\alpha_j$ is either zero or a unit in $R_0$ for $1\leq j\leq t$. Then $\alpha_j\neq 0$ for some $m\leq j\leq t$. We may assume that 
$j=m$ after rearrangements if needed. 
Then $b_m\in (a_1,\ldots,a_m,b_{m+1},\ldots,b_t)R \Rightarrow\mathfrak{J}=(a_1,\ldots,a_m,b_{m+1},\ldots,b_t)R$. 
In this case, $\deg a_m=g$ and \eqref{key-prop-2} holds for the choice of $a_m$.  
\end{proof}
Now, $\{a_1,\ldots,a_t\}$ is the desired basis of $\mathfrak{J}$.
\end{proof}
\begin{proposition}\label{proposition-basis-existence-2}
 Let $N$ be a Noetherian graded $R$-module of dimension $s>0$. Let 
$\mathfrak{J}_1,\ldots,\mathfrak{J}_m$ be ideals of $R$ generated by homogeneous elements of degrees
$g_1,\ldots,g_m$ respectively and $\dim N/\mathfrak{J}_iN=0$ for $1\leq i\leq m.$
Let $\mathfrak{B}_1$ be an 
$N$-basis of $\mathfrak{J}_1$ consisting of homogeneous elements of degree $g_1$. Then there exist $N$-bases $\mathfrak{B}_i$ of $\mathfrak{J}_i$ consisting of 
homogeneous elements of degree $g_i$ for $1\leq i\leq m$ such that any $s$ elements of the set $\mathop\bigcup\limits_{i=1}^m\mathfrak{B}_i$ form a
system of parameters of $N$. 
\end{proposition}
\begin{proof}
We apply induction on $m$. The case $m=1$ is trivial. Suppose that the result holds for $m-1$.  
Consider the family $\mathfrak{F}$ of quotient modules $N/JN$ where $J$ is the ideal generated by any $k$ elements of the set $\mathop\bigcup\limits_{i=1}^{m-1}\mathfrak{B}_i$
for $0\leq k\leq s$. Now by Proposition \ref{proposition-basis-existence-1}, there exists a system of homogeneous elements of degree 
$g_m$, say $\mathfrak{B}_m$, which forms $N/JN$-basis of $\mathfrak{J}_m$ for all $N/JN\in \mathfrak{F}$. 
This completes the proof.
\end{proof}
We now use the above results for $R(\I)$ and $G(\I)$ to obtain a generating set of $\N$ with the properties 
$(2)$-$(4)$ described in next lemma. 
We regard the Rees algebra $R(\I)=\mathop\oplus\limits_{n\geq 0}I_n$ as the $A$-subalgebra of 
 the polynomial ring $A[t]$. Let $\R_{\geq c}=\mathop\oplus\limits_{n\geq c}I_n$, $c\geq 1$ be the graded submodule of $R(\I)$ and
 $\beta$ be an integer such that $I_{n+1}=I_1I_n$ for all $n\geq \beta.$ Recall that $\nu:=\mu(\m)$.
\begin{lemma}\label{lemma-4.1-Y}
There exist elements $x_1,\ldots,x_\nu\in\m$ and $a_{ij}\in I_i$, $1\leq j\leq u_i$ for some integers $u_i\geq 0$ and for $1\leq i\leq \beta$  which satisfy the 
following conditions:
\begin{enumerate}
 \item\label{item-1} $x_1,\ldots,x_{\nu},a_{11}t,\ldots,a_{1u_1}t,\ldots,a_{ij}t^i,\ldots,a_{\beta u_{\beta}}t^\beta$ is a minimal generating set of 
 $\mathcal{N}$,
 \item\label{item-2} any $d$ elements from the set $\{a_{ij}t^i: 1\leq j\leq u_i,1\leq i\leq \beta\}$ is a system of parameters of $G(\I)$,
 \item\label{item-3} $x_1,\ldots,x_\nu$ is a minimal generating set of $\m$ and
 \item\label{item-4} any $d$ elements from $\{x_1,\ldots,x_\nu, a_{ij}: 1\leq j\leq u_i,1\leq i\leq \beta\}$ is a system of parameters of $A$.
 \end{enumerate}
\end{lemma}
\begin{proof}
Since $\R_{\geq \beta}$ has basis in degree $\beta$, 
there exists a $G(\I)$-basis, say $\mathfrak{B}_\beta=\{a_{\beta 1}t^\beta,\ldots,a_{\beta w_{\beta}}t^\beta\}$, of $\R_{\geq \beta}$ by 
Proposition \ref{proposition-basis-existence-1}. 
Let $b_{i 1},\ldots,b_{iv_i}$ be a minimal generating set of $I_i$ for $1\leq i\leq{\beta-1}$. Consider the ideal 
$\mathfrak{J}_{i}=(b_{i 1}t^i,\ldots,b_{iv_i}t^i)R(\I)$ for $1\leq i\leq{\beta-1}$. Then, 
$(\mathfrak{J}_i)_n= I_nt^n  \text{ for } n\gg 0$ 
which implies that $\dim G(\I)/\mathfrak{J}_iG(\I)=0$ for
$1\leq i\leq \beta-1$. By Proposition \ref{proposition-basis-existence-2}, there exist $G(I)$-bases, say $\mathfrak{B}_i=\{a_{i 1}t^i,\ldots,a_{iw_i}t^i\}$ of 
$\mathfrak{J}_i$, consisting of homogeneous elements of degree $i$ for all $1\leq i\leq \beta-1$ such that any $d$ elements from the set 
$\mathop\bigcup\limits_{i=1}^\beta\mathfrak{B}_i=\{ a_{ij}t^i ~|~ 1\leq j\leq w_i, 1\leq i\leq \beta\}$
is a system of parameters of $G(\I)$. Clearly, the set $\mathop\bigcup\limits_{i=1}^\beta\mathfrak{B}_i$ generates $\R_+$. We choose 
$\mathfrak{B}^\prime\subseteq \mathop\bigcup\limits_{i=1}^t\mathfrak{B}_i$ a minimal generating
set of $\R_+$. Let 
$$\mathfrak{B}^\prime=\{ a_{11}t,\ldots,a_{1u_1}t,\ldots,a_{ij}t^i,\ldots,a_{\beta u_{\beta}}t^\beta\}.$$ 
Then \eqref{item-2} holds. Further, any $d$ elements from $\{a_{ij}; 1\leq j\leq u_i, 1\leq i\leq \beta\}$ is a system of parameters of $A$.
Let $\mathfrak{F}$ be the family of the quotient modules $A/J$ where $J$ is an ideal generated by any $k$ elements of the set 
$\{a_{ij}: 1\leq j\leq u_i,1\leq i\leq \beta\}$ for 
$0\leq k\leq d$. By \cite[Proposition 1.9, Chap I]{vogel}, there exists an $A/J$-basis $x_1,\ldots,x_\nu$ of $\m$ for each $A/J\in\mathfrak{F}$. Thus
we obtain \eqref{item-1}-\eqref{item-4}.
\end{proof}
\section{Buchsbaumness of $G(\I)$}\label{section-main-proof}
In this section, we discuss the proof of Theorem \ref{theorem-intro-2}. Suppose $A$ is a Buchsbaum local ring and $Q=(a_1,\ldots,a_d)$ is a 
reduction of $\I$ with $I_{n+1}=QI_n$ for $n\geq r$. Then \eqref{main-theorem-item-1} implies \eqref{main-theorem-item-2} is obvious and the
equivalence of \eqref{main-theorem-item-2} and \eqref{main-theorem-item-3} follows from Proposition 
\ref{prop-I-equi-1} since $a_1^{n_1},\ldots,a_d^{n_d}$ is an unconditioned d-sequence for all $n_i\geq 1$, see Remark \ref{remark-from-trung-huneke}.
The aim of this section is to prove Theorem \ref{main-theorem-last-section} stated below
which, along with Proposition  \ref{proposition-3.6-Y}, completes the proof of \eqref{main-theorem-item-2} $\Rightarrow$ \eqref{main-theorem-item-1}.
An important ingredient of our proof is Lemma \ref{lemma-4.1-Y} which provides a desired
generating set of $\N$ for defining Koszul complex on $\N$. 
We use \cite[Theorem 2.15, Chap I]{vogel} for our proof which provides a sufficient condition for the Buchsbaumness of $G(\I)$
using canonical maps between the Koszul (co)homology modules and the local cohomology modules of $G(\I)$. 

We first fix some notations. Let $K^\cdotp(\mathcal{N};G(\I))$ denote the Koszul (co)complex on a minimal generating set of the ideal $\mathcal{N}$.
Let $x_1,\ldots,x_\nu\in\m$ and $a_{ij}t^i\in R(\I)_i$ for $1\leq j\leq u_i$, $1\leq i\leq \beta$ be a minimal generating set of $\mathcal{N}$ such 
that all the conditions in Lemma \ref{lemma-4.1-Y} are satisfied. Then $K^\cdotp(\mathcal{N};G(\I))$ is the Koszul complex 
 on the above system of elements up to isomorphism and we have 
 \[ K^\cdotp(\N;G(\I))=\smashoperator[r]{\mathop\bigoplus_{\substack{\Gamma\subseteq [1,v], \Lambda_j\subseteq [1,u_j],\\ 1\leq j\leq \beta}}} G(\I)e_\Gamma^{\Lambda_1\ldots \Lambda_{\beta}} 
\text{ with }  K^i(\N;G(\I))=\smashoperator[r]{\mathop\bigoplus_{\substack{|\Gamma|+|\Lambda_1|+\ldots+|\Lambda_{\beta}|=i,\\ 
\Gamma\subseteq [1,v], \Lambda_j\subseteq [1,u_j], 1\leq j\leq \beta}}} G(\I)e_\Gamma^{\Lambda_1 \Lambda_2\ldots \Lambda_{\beta}} \]
 where $\{e_\Gamma^{\Lambda_1 \Lambda_2\ldots \Lambda_{\beta}} ~|~\Gamma\subseteq [1,v] \text{ and } \Lambda_j\subseteq [1,u_j] \text{ for } 1\leq j\leq \beta\}$ is the graded free basis 
 with $\deg e_\Gamma^{\Lambda_1\Lambda_2\ldots \Lambda_{\beta}}=-(|\Lambda_1|+2|\Lambda_2|+\ldots+\beta|\Lambda_\beta|).$ We write $\h^i\big(\mathcal{N};G(\I)\big)$ for the homology modules 
 $\h^i\big(K^\cdotp(\mathcal{N};G(\I))\big)$. The notations $\underline{x(n)}$ and $\underline{a_it^i(n)}$ are used for the 
 sequences $x_1^n,\ldots,x_{\nu}^n$ and $(a_{i1}t^i)^n,\ldots,(a_{iu_i}t^i)^n$ respectively for $1\leq i\leq\beta$. 
 \begin{theorem}\label{main-theorem-last-section}
  Suppose that $A$ is a Buchsbaum local ring  and the equality $\mathds{I}(G(\I))=\mathds{I}(A)$ holds. Then the associated graded ring $G(\I)$ is Buchsbaum over $R(\I)$. 
 \end{theorem}
\begin{proof}
Suppose $A$ is Buchsbaum and $\mathds{I}(G(\I))=\mathds{I}(A)$. Then $G(\I)$ is quasi-Buchsbaum by Theorem \ref{theorem-quasi-buchs}. We prove that 
$G(\I)$ is Buchsbaum by induction on $d$. We may assume that $d\geq 2$.
By \cite[Theorem 2.15, Chap I]{vogel}, it is enough to prove that the canonical map 
\[\phi^i_{G(\I)}: \h^i\big(\mathcal{N};G(\I)\big)\longrightarrow \varinjlim \h^i\big(\underline{x(n)},\underline{a_it^i(n)};G(\I)\big)\cong \h_{\mathcal{N}}^i(G(\I))\]
is surjective for all $0\leq i< d$.  

Since any system of parameters is a d-sequence in $A$, using Proposition \ref{proposition-2.2-Y} and Proposition \ref{prop-I-equi-1}, 
we get that \begin{equation} 
             \label{eq-main-proof-1}G(\I A^\prime)\cong G(\I)/fG(\I)\end{equation} 
             where $A^\prime=A/(a_{11}), f=a_{11}t$ and $\I A^\prime$ denote the 
filtration $\{I_n+(a_{11})/(a_{11})\}_{n\geq 0}$ of $A^\prime$.
Since $A$ and $G(\I)$ are quasi-Buchsbaum, $\m \h^i_\m(A)=0$ and $\N\h^i_{\N}(G(\I))=0$ for $0\leq i<d$. 
Then it is easy to see that the equality $\mathds{I}(A^\prime)=\mathds{I}(G(\I)/fG(\I))=\mathds{I}(G(\I A^\prime))$ holds, see \cite[Lemma 1.7]{trung-gcm}. 
Therefore, by induction hypothesis, $G(\I A^\prime)$ is 
Buchsbaum. 

Now suppose $\depth A>0$. Then $a_{11}$ is a non-zero-divisor. In view of \eqref{eq-main-proof-1},  
$f$ is a non-zero-divisor on $G(\I)$.  The exact sequence 
\[0\longrightarrow G(\I)(-1)\xrightarrow{f} G(\I)\to G(\I  A^\prime)\longrightarrow 0\]
of graded $R(\I)$-modules gives the following commutative diagram 
\[\begin{tikzcd}
H^{i-1}(\mathcal{N};G(\I A^\prime)) \arrow{r} \arrow[swap]{d}{\phi^{i-1}_{G(\I A^\prime)}} & H^{i}(\mathcal{N};G(\I))(-1) \arrow{r}{f} \arrow{d}{\phi^{i}_{G(\I)(-1)}} & 0 \\
\h_{\mathcal{N}}^{i-1}(G(\I A^\prime)) \arrow{r} & \h_{\mathcal{N}}^{i}(G(\I))(-1) \arrow{r}{f}  & 0
\end{tikzcd}
\]
where the rows are exact and $\phi^{i-1}_{G(\I A^\prime)}$ is surjective for $0\leq i-1< d-1$. This implies the surjectivity of $\phi^{i}_{G(\I)}$
for $0<i<d$. Thus $G(\I)$ is a Buchsbaum module over $R(\I)$. 

Now suppose $\depth A=0.$ Then, by Proposition \ref{proposition-3.6-Y}, we have that $\h_N^0(G(\I))\neq 0$, $\mathds{I}(A/U)=\mathds{I}(G(\I A/U))$ and the following
is a short exact sequence  
\begin{align}\label{eq-main-proof-2} 0\longrightarrow \h^0_{\mathcal{N}}(G(\I))\longrightarrow G(\I)\longrightarrow G(\I A/U)\longrightarrow 0
 \end{align}
of graded $R(\I)$-modules. Since $A/U$ is Buchsbaum and $\depth A/U>0$, $G(\I A/U)$ is Buchsbaum by the previous case.
So, $\phi^i_{G(\I A/U)}$ is surjective for $0\leq i<d$. In order to prove the surjectivity of $\phi^i_{G(\I)}$, it is enough to show that 
the canonical maps 
\[\tau^i_A:\h^i(\mathcal{N};\h^0_{\mathcal{N}}(G(\I)))\longrightarrow \h^i(N;G(\I))\]
 are injective for $0\leq i\leq d$, see proof of \cite[Theorem 2.15, Chap I]{vogel}. 
First, we discuss the case $0\leq i\leq d-1$. Consider the following commutative diagram with canonical maps:
\[\begin{tikzcd}
\h^0_{\mathcal{N}}(G(\I)) \arrow{r}{\tau_A} \arrow[swap]{d}{\sigma} & G(\I)\arrow{d} \\
\h^0_{\mathcal{N}}(G(\I A^\prime))  \arrow{r}{\tau_{A^\prime}} & G(\I A^\prime)
\end{tikzcd}
\]
It follows from Remark \ref{remark-from-trung-huneke} and Proposition \ref{prop-I-equi-1} that $a_{11}t,\ldots,a_{1d}t$ is a d-sequence on $G(\I)$
and $\mathds{I}(G(\I))=\mathds{I}((a_{11}t,\ldots,a_{1d}t);G(\I)).$ Therefore $\h^0_{\mathcal{N}}(G(\I))\cap fG(\I)=0$ which implies that 
$\tau_{A^\prime}\circ \sigma$ is injective. Since $\h^0_{\mathcal{N}}(G(\I))$ and $\h^0_{\mathcal{N}}(G(\I A^\prime))$ are vector spaces over $R(\I)/\mathcal{N}$, the vertical map 
$\sigma$ splits. Therefore,  $\sigma^i$ splits in the following commutative diagram:
\[\begin{tikzcd}
H^{i}(\mathcal{N};\h^0_{\mathcal{N}}(G(\I))) \arrow{r}{\tau^i_A} \arrow[swap]{d}{\sigma^i} & H^{i}(\mathcal{N};G(\I))\arrow{d} \\
H^{i}(\mathcal{N}; \h^0_{\mathcal{N}}(G(\I A^\prime)))  \arrow{r}{\tau^i_{A^\prime}} & H^{i}(\mathcal{N};G(\I A^\prime))
\end{tikzcd}
\]
By induction hypothesis, the map $\tau^i_{A^\prime}$ is injective for $0\leq i\leq d-1$. Hence $\tau^i_A$ is injective for $0\leq i\leq d-1$. 
Now we prove the injectivity of 
   \[\tau^d_A:H^d(\mathcal{N}; \h^0_{\mathcal{N}}(G(\I))) \longrightarrow H^d(\mathcal{N};G(\I)).\]
By \eqref{eq-main-proof-2}, we have the following exact sequence of graded Koszul complexes
\begin{equation}\label{eq-main-proof-3} 0\to K^\cdotp(\N;\h^0_\N(G(\I)))\to K^\cdotp(\N;G(\I))\to K^\cdotp(\N;G(\I A/U))\to 0.\end{equation}
Since $\N\h^0_\N(G(\I))=0$, all the differentials of the Koszul complex $K^\cdotp(\N;\h^0_\N(G(\I)))$ are zero maps. We get the following commutative diagram:
\[\begin{tikzcd}
0 \arrow{r}&  K^{d-1}(\mathcal{N};\h^0_{\mathcal{N}}(G(\I))) \arrow{r}{} \arrow[swap]{d}{0} & K^{d-1}(\mathcal{N};G(\I))\arrow{d}{\partial} \\
0\arrow{r}&  K^{d}(\mathcal{N}; \h^0_{\mathcal{N}}(G(\I)))  \arrow{r}{} & K^d(\mathcal{N};G(\I)))
\end{tikzcd}
\]
Let $\xi\in K^{d}(\mathcal{N}; \h^0_{\mathcal{N}}(G(\I)))$ be a homogeneous element with $\deg \xi=n$ such that there exists a homogeneous element 
$\eta\in  K^{d-1}(\mathcal{N};G(\I))$ with $\xi=\partial(\eta)\in K^d(\mathcal{N};G(\I)))$. We show that $\xi=0$. Then it follows that $\tau^d_A$ is injective.
Write $\xi$ and $\eta$ as follows: 
\[\xi=\smashoperator[r]{\sum_{\substack{|\Gamma|+|\Lambda_1|+\ldots+|\Lambda_\beta|=d,\\ \Gamma\subseteq [1,\nu], \Lambda_j\subseteq [1,u_j], 1\leq j\leq \beta}}}\xi_\Gamma^{\Lambda_1\ldots \Lambda_\beta} e_\Gamma^{\Lambda_1\ldots \Lambda_\beta} 
\quad \textrm{and} \quad \eta=\smashoperator[r]{\sum_{\substack{|Q|+|P_1|+\ldots+|P_\beta|=d-1,\\ Q\subseteq [1,\nu], P_j\subseteq [1,u_j],1\leq j\leq \beta}}}\eta_Q^{P_1\ldots P_\beta} e_Q^{P_1\ldots P_\beta}\]
where $ \xi_\Gamma^{\Lambda_1\ldots \Lambda_\beta} \in[G(\I)]_{n+|\Lambda_1|+2|\Lambda_2|+\ldots+\beta|\Lambda_\beta|}$ and 
$\eta_Q^{P_1\ldots P_\beta}\in [G(\I)]_{n+|P_1|+2|P_2|+\ldots+\beta|P_\beta|}$. 
Since $\xi=\partial(\eta)$, we get 
\begin{align}\xi_\Gamma^{\Lambda_1\ldots \Lambda_\beta}=\sum_{j\in \Gamma} (-1)^{\Gamma(j)}x_j \eta_{\Gamma\setminus j}^{\Lambda_1\ldots \Lambda_\beta}+\sum_{k=1}^\beta\sum_{i\in \Lambda_k}(-1)^{|\Gamma|+|\Lambda_1|+\ldots+|\Lambda_{k-1}|+\Lambda_k(i)} a_{ki}t^k \eta_{\Gamma}^{\Lambda_1\ldots \Lambda_{k-1}\Lambda_k\setminus\{i\}\Lambda_{k+1}\ldots \Lambda_\beta}\label{new-eqn-v5-1}\end{align}
where $\Gamma(j):=|\{j^\prime\in \Gamma| j^\prime <j\}|$ and $\Lambda_k(i):=|\{i^\prime\in \Lambda_k|i^\prime <i\}|.$ 
Note that $\xi_\Gamma^{\Lambda_1\ldots \Lambda_\beta}\in \h^0_{\mathcal{N}}(G(\I))$.
We show that $\xi_\Gamma^{\Lambda_1\ldots \Lambda_\beta}=0$  for all 
$\Gamma,\Lambda_1,\ldots,\Lambda_\beta$ with $|\Gamma|+|\Lambda_1|+\ldots+|\Lambda_\beta|=d$.

Let $c_Q^{P_1,\ldots,P_\beta}\in I_{n+|P_1|+2|P_2|+\ldots+\beta|P_\beta|}$ be a representative of $\eta_Q^{P_1,\ldots,P_\beta}$, i.e., 
$\overline{c_Q^{P_1,\ldots,P_\beta}}= \eta_Q^{P_1,\ldots,P_\beta}$ in $[G(\I)]_{n+|P_1|+2|P_2|+\ldots+\beta|P_\beta|}$.
Then the element
\begin{align}
 b_\Gamma^{\Lambda_1\ldots \Lambda_\beta}:={}&\sum_{j\in \Gamma} (-1)^{\Gamma(j)}x_j c_{\Gamma\setminus\{j\}}^{\Lambda_1\ldots \Lambda_\beta}+\sum_{k=1}^\beta\sum_{i\in \Lambda_k}(-1)^{|\Gamma|+|\Lambda_1|+\ldots+|\Lambda_{k-1}|+\Lambda_k(i)}a_{ki} c_{\Gamma}^{\Lambda_1\ldots \Lambda_{k-1}\Lambda_k\setminus\{i\}\Lambda_{k+1}\ldots \Lambda_\beta}\label{eq-main-proof-4}\\ 
&\in{} I_{n+|\Lambda_1|+2|\Lambda_2|+\ldots+\beta|\Lambda_\beta|}\nonumber
 \end{align}
 and $\overline{b_\Gamma^{\Lambda_1\ldots \Lambda_\beta}}=\xi_\Gamma^{\Lambda_1\ldots \Lambda_\beta}$ in $[G(\mathcal{I})]_{n+|\Lambda_1|+2|\Lambda_2|+\ldots+\beta|\Lambda_\beta|}$. This representation $b_{\Gamma}^{\Lambda_1\ldots \Lambda_{\beta}}$ of $\xi_\Gamma^{\Lambda_1\ldots \Lambda_\beta}$ depends on the choice of representations $c_Q^{P_1,\ldots,P_\beta}$
 of $\eta_Q^{P_1\ldots P_\beta}$ in \eqref{new-eqn-v5-1}. Using the following lemma,  
 we will conclude that $b_\Gamma^{\Lambda_1\ldots \Lambda_\beta}=0$ after a suitable change of representations $c_Q^{P_1\ldots P_\beta}$. Thus 
 $\xi_\Gamma^{\Lambda_1\ldots \Lambda_\beta}=0$. 
\begin{lemma}\label{Lemma-inside-main-proof} 
Let $\Gamma\subseteq [1,v], \Lambda_j\subseteq [1,u_j], 1\leq j\leq \beta$ with $|\Gamma|+|\Lambda_1|+\ldots+|\Lambda_\beta |=d$ and $b_\Gamma^{\Lambda_1\ldots \Lambda_\beta}\in I_{n+|\Lambda_1|+2|\Lambda_2|+\ldots+\beta|\Lambda_\beta|}$ be 
same as described in \eqref{eq-main-proof-4} such that $\overline{b_\Gamma^{\Lambda_1\ldots \Lambda_\beta}}=\xi_\Gamma^{\Lambda_1\ldots \Lambda_\beta}$. Suppose that there exists a subset $\Lambda_1^\prime$ of $[1,u_1]$ such that the following conditions are satisfied:
 \begin{enumerate}
  \item $\Lambda_1 \subseteq  \Lambda_1^\prime $ and $|\Lambda_1^\prime|=|\Lambda_1|+1,$
  \item $b_{\Gamma^\prime}^{\Lambda_1^\prime \Lambda_2\ldots \Lambda_k}=0$ for all $\Gamma^\prime\subseteq \Gamma$ such that $|\Gamma^\prime|=|\Gamma|-1$ and 
  \item $b_\Gamma^{\Lambda_1^\prime \Lambda_2\ldots \Lambda_{k-1}\Lambda_k^\prime \Lambda_{k+1} \Lambda_\beta}=0$ for all $\Lambda_k^\prime\subseteq \Lambda_k$ such that $|\Lambda_k^\prime|=|\Lambda_k|-1$ for 
  $2\leq k\leq \beta$.
 \end{enumerate}
Then, $b_\Gamma^{\Lambda_1\Lambda_2\ldots \Lambda_\beta}=0$ after a suitable change of representations $c_\Gamma^{\Lambda_1\setminus \{i\}\Lambda_2\ldots \Lambda_\beta}$ 
for $i\in \Lambda_1$ in \eqref{eq-main-proof-4}.
\end{lemma}
\begin{proof}[Proof of Lemma]
 First consider the case when $|\Lambda_1|=d$. Then $\Gamma=\Lambda_k=\phi$ for $2\leq k\leq \beta$ and 
 \[\xi_\Gamma^{\Lambda_1\ldots \Lambda_\beta}=\mathop\sum\limits_{i\in \Lambda_1}(-1)^{\Lambda_1(i)}a_{1i}t\eta_\Gamma^{\Lambda_1\setminus\{i\}\Lambda_2\ldots \Lambda_\beta}\in (a_{1i}t:i\in \Lambda_1)G(\I)\cap \h^0_{\N}(G(\I))=0\]
since $\{a_{1i}t:i\in \Lambda_1\}$ is a d-sequence on $G(\I)$ by Proposition \ref{prop-I-equi-1}. Thus 
\[b_\Gamma^{\Lambda_1\ldots \Lambda_\beta}= \mathop\sum\limits_{i\in \Lambda_1}(-1)^{\Lambda_1(i)}a_{1i}c_\Gamma^{\Lambda_1\setminus\{i\}\Lambda_2\ldots \Lambda_\beta}\in (a_{1i}:i\in \Lambda_1)\cap I_{n+|\Lambda_1|+1}=(a_{1i}:i\in \Lambda_1) I_{n+|\Lambda_1|}\]
 where the last equality follows from Proposition \ref{lemma-5.1/trung} and the fact that 
 \[\mathds{I}(A)=\mathds{I}(G(\I))\geq \mathds{I}((a_{11}t,\ldots,a_{1d}t);G(\I))\geq \mathds{I}((a_{11},\ldots,a_{1d});A)=\mathds{I}(A).\] Let $g_i\in I_{n+|\Lambda_1|}$ such that
$b_\Gamma^{\Lambda_1,\ldots,\Lambda_\beta}=\mathop\sum\limits_{i\in \Lambda_1}a_{1i}g_i.$  Then, for each $i\in \Lambda_1$, 
\[\eta_\Gamma^{\Lambda_1\setminus\{i\}\Lambda_2\ldots \Lambda_\beta}=\overline{c_\Gamma^{\Lambda_1\setminus\{i\}\Lambda_2\ldots \Lambda_\beta}}=\overline{c_\Gamma^{\Lambda_1\setminus\{i\}\Lambda_2\ldots \Lambda_\beta}-g_i}\]
in $G(\I)_{n+|\Lambda_1|-1}.$ Therefore the assertion follows after replacing $c_\Gamma^{\Lambda_1\setminus\{i\}\Lambda_2\ldots \Lambda_\beta}$ by 
$c_\Gamma^{\Lambda_1\setminus\{i\}\Lambda_2\ldots \Lambda_\beta}-g_i$, for each $i\in \Lambda_1$, for the representative of $\eta_\Gamma^{\Lambda_1\setminus\{i\}\Lambda_2\ldots \Lambda_\beta}$ in the 
expression \eqref{eq-main-proof-4} of $b_\Gamma^{\Lambda_1,\ldots,\Lambda_\beta}$.

Now, suppose $|\Lambda_1|<d$. Then $\Gamma\neq \phi$ or $\Lambda_k\neq \phi$ for some $2\leq k\leq \beta$. It is enough to consider the case when  $\Gamma\neq \phi$. 
Similar arguments will hold when $\Lambda_k\neq \phi$ for some $k$. 
For an integer $k$, we write $S(k)$ for the sum $|\Lambda_1|+\ldots+|\Lambda_k|$.
Now, for each $j\in \Gamma,$ we have that 
\begin{align*}
 0=b_{\Gamma\setminus \{j\}}^{\Lambda_1^\prime \Lambda_2\ldots \Lambda_\beta}=&\sum_{j^\prime\in \Gamma\setminus \{j\}} (-1)^{\Gamma\setminus\{j\}(j^\prime)}x_{j^\prime} c_{\Gamma\setminus\{j,j^\prime\}}^{\Lambda_1^\prime \Lambda_2\ldots \Lambda_\beta}
+ \sum_{i\in \Lambda_1^\prime}(-1)^{|\Gamma|-1+\Lambda_1^\prime(i)} a_{1i} c_{\Gamma\setminus \{j\}}^{\Lambda_1^\prime\setminus\{i\}\Lambda_2\ldots \Lambda_\beta}\nonumber\\
 &+ \sum_{k=2}^\beta\sum_{i\in \Lambda_k}(-1)^{|\Gamma|-1+S(k-1)+1+\Lambda_k(i)}a_{ki} c_{\Gamma\setminus\{j\}}^{\Lambda_1^\prime \Lambda_2\ldots \Lambda_{k-1}\Lambda_k\setminus\{i\}\Lambda_{k+1}\ldots \Lambda_\beta}\\
\end{align*}
where the notation $\Gamma\setminus\{j\}(j^\prime)$ is used for the cardinality $|\{j^{\prime\prime}\in \Gamma\setminus\{j\} |j^{\prime\prime}< j^\prime \}|=|\{j^{\prime\prime}\in \Gamma| j\neq j^{\prime\prime}< j^\prime \}|
$.
Multiplying by $(-1)^{\Gamma(j)}x_j$ and taking the sum $\mathop\sum\limits_{j\in \Gamma}$, we get that
\begin{align}
 0={}&\sum_{j\in \Gamma} \sum_{j^\prime\in \Gamma\setminus \{j\}}(-1)^{\Gamma(j)}(-1)^{\Gamma\setminus\{j\}(j^\prime)}x_jx_{j^\prime} c_{\Gamma\setminus\{j,j^\prime\}}^{\Lambda_1^\prime \Lambda_2\ldots \Lambda_\beta}
 +\sum_{i\in \Lambda_1^\prime}(-1)^{|\Gamma|-1+\Lambda_1^\prime(i)} a_{1i} \sum_{j\in \Gamma} (-1)^{\Gamma(j)} x_jc_{\Gamma\setminus \{j\}}^{\Lambda_1^\prime\setminus\{i\}\Lambda_2\ldots \Lambda_\beta}\nonumber\\
&+ \sum_{k=2}^\beta\Big(\sum_{i\in \Lambda_k}(-1)^{|\Gamma|+S(k-1)+\Lambda_k(i)}a_{ki} \sum_{j\in \Gamma} (-1)^{\Gamma(j)} x_jc_{\Gamma\setminus\{j\}}^{\Lambda_1^\prime \Lambda_2\ldots \Lambda_{k-1}\Lambda_k\setminus\{i\}\Lambda_{k+1}\ldots \Lambda_\beta}\Big)\nonumber\\
\begin{split}\label{last-part-lemma-eqn-1}
={}& \sum_{i\in \Lambda_1^\prime}(-1)^{|\Gamma|-1+\Lambda_1^\prime(i)} a_{1i} \sum_{j\in \Gamma} (-1)^{\Gamma(j)} x_jc_{\Gamma\setminus \{j\}}^{\Lambda_1^\prime\setminus\{i\}\Lambda_2\ldots \Lambda_\beta}\\
& + \sum_{k=2}^\beta\Big(\sum_{i\in \Lambda_k}(-1)^{|\Gamma|+S(k-1)+\Lambda_k(i)}a_{ki} \sum_{j\in \Gamma} (-1)^{\Gamma(j)} x_jc_{\Gamma\setminus\{j\}}^{\Lambda_1^\prime \Lambda_2\ldots \Lambda_{k-1}\Lambda_k\setminus\{i\}\Lambda_{k+1}\ldots \Lambda_\beta}\Big)
\end{split}
\end{align}
as the first term $\mathop\sum\limits_{j\in \Gamma} \sum\limits_{j^\prime\in \Gamma\setminus \{j\}}(-1)^{\Gamma(j)}(-1)^{\Gamma\setminus\{j\}(j^\prime)}x_jx_{j^\prime} c_{\Gamma\setminus\{j,j^\prime\}}^{\Lambda_1^\prime \Lambda_2\ldots \Lambda_\beta}=0$.
Further, for each  $i\in \Lambda_k$, $2\leq k\leq \beta$, we have that
\begin{align}
0={}&b_{\Gamma}^{\Lambda_1^\prime \Lambda_2\ldots \Lambda_{k-1}\Lambda_k\setminus \{i\} \Lambda_{k+1}\ldots \Lambda_\beta}\nonumber\\
\begin{split}\label{last-part-lemma-eqn-2}
={}& \sum_{j\in \Gamma} (-1)^{\Gamma(j)}x_{j} c_{\Gamma\setminus\{j\}}^{\Lambda_1^\prime \Lambda_2\ldots \Lambda_{k-1}\Lambda_k\setminus \{i\}\Lambda_{k+1}\ldots \Lambda_\beta}
 +\sum_{i^\prime\in \Lambda_1^\prime}(-1)^{|\Gamma|+\Lambda_1^\prime(i^\prime)} a_{1i^\prime} c_{\Gamma}^{\Lambda_1^\prime\setminus\{i^\prime\}\Lambda_2\ldots \Lambda_{k-1} \Lambda_k\setminus\{i\} \Lambda_{k+1}\ldots \Lambda_\beta}\\
&+ \sum_{k^\prime=2}^{k-1}\sum_{i^\prime\in \Lambda_{k^\prime}}(-1)^{|\Gamma|+S(k^\prime-1)+1+\Lambda_{k^\prime}(i^\prime)}a_{k^\prime i^\prime}c_{\Gamma}^{\Lambda_1^\prime \Lambda_2\ldots \Lambda_{k^\prime-1}\Lambda_{k^\prime}\setminus\{i^\prime\}\Lambda_{k^\prime+1}\ldots \Lambda_{k}\setminus\{i\} \Lambda_{k+1}\ldots \Lambda_\beta}\\
&+\sum_{i^\prime\in \Lambda_{k}\setminus\{i\}}(-1)^{|\Gamma|+S(k-1)+1+\Lambda_k\setminus\{i\}(i^\prime)}a_{ki^\prime} c_{\Gamma}^{\Lambda_1^\prime \Lambda_2\ldots \Lambda_{k-1}\Lambda_k\setminus\{i,i^\prime\}\Lambda_{k+1}\ldots \Lambda_\beta}\\
&+ \sum_{k^\prime=k+1}^\beta\sum_{i^\prime\in \Lambda_{k^\prime}}(-1)^{|\Gamma|+S(k^\prime-1)+\Lambda_{k^\prime}(i^\prime)} a_{k^\prime i^\prime} c_{\Gamma}^{\Lambda_1^\prime \Lambda_2\ldots \Lambda_{k-1}\Lambda_k\setminus\{i\}\Lambda_{k+1}\ldots \Lambda_{k^\prime-1} \Lambda_{k^\prime}\setminus\{i^\prime\}\Lambda_{k^\prime+1}\ldots \Lambda_\beta}.
\end{split}
\end{align}
Again, we multiply the last three terms by $\mathop\sum\limits_{i\in \Lambda_k}(-1)^{|\Gamma|+S(k-1)+\Lambda_k(i)}a_{ki}$ and take the sum $\sum_{k=2}^\beta$ to obtain the following
expression which is zero.

{\it{Claim 2.}}  We have that 
\begin{align*}
 &\sum_{k=2}^\beta\Big(\sum_{i\in \Lambda_k}(-1)^{|\Gamma|+S(k-1)+\Lambda_k(i)}a_{ki}\Big((-1)^{|\Gamma|}\sum_{i^\prime\in \Lambda_{k}\setminus \{i\}}(-1)^{S(k-1)+1+\Lambda_k\setminus\{i\}(i^\prime)}a_{ki^\prime} c_{\Gamma}^{\Lambda_1^\prime \Lambda_2\ldots \Lambda_{k-1}\Lambda_k\setminus\{i,i^\prime\}\Lambda_{k+1}\ldots \Lambda_\beta}\nonumber\\
 &+ \sum_{k^\prime=2}^{k-1}\sum_{i^\prime\in \Lambda_{k^\prime}}(-1)^{S(k^\prime-1)+1+\Lambda_{k^\prime}(i^\prime)}a_{k^\prime i^\prime}c_{\Gamma}^{\Lambda_1^\prime \Lambda_2\ldots \Lambda_{k^\prime-1}\Lambda_{k^\prime}\setminus\{i^\prime\}\Lambda_{k^\prime+1}\ldots \Lambda_{k}\setminus\{i\} \Lambda_{k+1}\ldots \Lambda_\beta}\nonumber\\
&+\sum_{k^\prime=k+1}^\beta\sum_{i^\prime\in \Lambda_{k^\prime}}(-1)^{S(k^\prime-1)+\Lambda_{k^\prime}(i^\prime)} a_{k^\prime i^\prime} c_{\Gamma}^{\Lambda_1^\prime \Lambda_2\ldots \Lambda_{k-1}\Lambda_k\setminus\{i\}\Lambda_{k+1}\ldots \Lambda_{k^\prime-1} \Lambda_{k^\prime}\setminus\{i^\prime\}\Lambda_{k^\prime+1}\ldots \Lambda_\beta}
\Big)\Big)=0.
\end{align*}
\begin{proof}[Proof of Claim 2] First note that, 
\begin{align*}
&\sum_{i\in \Lambda_k}(-1)^{|\Gamma|+S(k-1)+\Lambda_k(i)}a_{ki} \sum_{i^\prime\in \Lambda_{k}\setminus \{i\}}(-1)^{|\Gamma|+S(k-1)+1+\Lambda_k\setminus\{i\}(i^\prime)}a_{ki^\prime} c_{\Gamma}^{\Lambda_1^\prime \Lambda_2\ldots \Lambda_{k-1}\Lambda_k\setminus\{i,i^\prime\}\Lambda_{k+1}\ldots \Lambda_\beta}\\
={}&\sum_{i\in \Lambda_k}\mathop\sum\limits_{i^\prime\in \Lambda_{k}\setminus \{i\}}(-1)^{\Lambda_k(i)+1+\Lambda_k\setminus\{i\}(i^\prime)}a_{ki} a_{ki^\prime} c_{\Gamma}^{\Lambda_1^\prime \Lambda_2\ldots \Lambda_{k-1}\Lambda_k\setminus\{i,i^\prime\}\Lambda_{k+1}\ldots \Lambda_\beta}\\
={}&0
\end{align*}
for each $2\leq k\leq \beta$. Further, let $2\leq l< m\leq \beta$ and $i\in \Lambda_l, j\in \Lambda_m$. Then the coefficients of $a_{li}a_{mj}$ in
the expression of  Claim 2 is 
\begin{align*}
 ((-1)^{2|\Gamma|+S(l-1)+\Lambda_l(i)+S(m-1)+\Lambda_m(j)}+(-1)^{2|\Gamma|+S(m-1)+\Lambda_m(j)+S(l-1)+1+\Lambda_l(i)})c=0
\end{align*}
where $c=c_\Gamma^{\Lambda_1^\prime \Lambda_2\ldots \Lambda_{l-1}\Lambda_l\setminus\{i\}\Lambda_{l+1}\ldots \Lambda_{m-1}\Lambda_m\setminus\{j\}\Lambda_{m+1}\ldots \Lambda_\beta}$.
This completes the proof of  Claim 2.
\end{proof}
Suppose $i_0\in\Lambda_1^\prime$ such that $\Lambda_1=\Lambda_1^\prime\setminus\{i_0\}$.  Putting the value of $\sum_{j\in \Gamma} (-1)^{\Gamma(j)}x_{j} c_{\Gamma\setminus\{j\}}^{\Lambda_1^\prime \Lambda_2\ldots \Lambda_{k-1}\Lambda_k\setminus \{i\}\Lambda_{k+1}\ldots \Lambda_\beta}$ 
from \eqref{last-part-lemma-eqn-2} into \eqref{last-part-lemma-eqn-1} for each $2\leq k\leq \beta$ and using Claim 2, we get that 
\begin{align}
 &(-1)^{|\Gamma|-1+\Lambda_1^\prime(i_0)} a_{1i_0}\Big( \sum_{j\in \Gamma} (-1)^{\Gamma(j)} x_jc_{\Gamma\setminus \{j\}}^{\Lambda_1\ldots \Lambda_\beta}+ 
  \sum_{k=2}^\beta\sum_{i\in \Lambda_k}(-1)^{|\Gamma|+S(k-1)+\Lambda_k(i)}a_{ki}  c_{\Gamma}^{\Lambda_1\ldots \Lambda_{k-1} \Lambda_k\setminus\{i\} \Lambda_{k+1}\ldots \Lambda_\beta}  \Big)\nonumber\\
=& \sum_{k=2}^\beta\Big(\sum_{i\in \Lambda_k}(-1)^{|\Gamma|+S(k-1)+\Lambda_k(i)}a_{ki}\sum_{i^\prime\in \Lambda_1^\prime, i^\prime\neq i_0}(-1)^{|\Gamma|+\Lambda_1^\prime(i^\prime)} a_{1i^\prime} c_{\Gamma}^{\Lambda_1^\prime\setminus\{i^\prime\}\Lambda_2\ldots \Lambda_{k-1} \Lambda_k\setminus\{i\} \Lambda_{k+1}\ldots \Lambda_\beta}\Big)\nonumber\\
 &-\sum_{i\in \Lambda_1^\prime, i\neq i_0}(-1)^{|\Gamma|-1+\Lambda_1^\prime(i)} a_{1i} \sum_{j\in \Gamma} (-1)^{\Gamma(j)} x_jc_{\Gamma\setminus \{j\}}^{\Lambda_1^\prime\setminus\{i\}\Lambda_2\ldots \Lambda_\beta} \nonumber\\
\in& (a_{1i}:i\in \Lambda_1)A.\label{eee-1}
\end{align}
Since $A$ is Buchsbaum,
$(a_{1i}: i\in \Lambda_1): a_{1i_0}=(a_{1i}: i\in \Lambda_1): \m\subseteq (a_{1i}: i\in \Lambda_1): x_{j^\prime}$ for $j^\prime\in \Gamma$.
We rewrite the expression in \eqref{eq-main-proof-4} and use \eqref{eee-1} to get 
\begin{align*}
 b_{\Gamma}^{\Lambda_1\ldots,\Lambda_\beta} -\sum_{i\in \Lambda_1}(-1)^{|\Gamma|+\Lambda_1(i)}a_{1i}  c_{\Gamma}^{\Lambda_1\setminus\{i\} \Lambda_{2}\ldots \Lambda_\beta}={}&
 \sum_{j\in \Gamma} (-1)^{\Gamma(j)} x_jc_{\Gamma\setminus \{j\}}^{\Lambda_1\Lambda_2\ldots \Lambda_\beta}\\ 
   &+ \sum_{k=2}^\beta\sum_{i\in \Lambda_k}(-1)^{|\Gamma|+S(k-1)+\Lambda_k(i)}a_{ki}  c_{\Gamma}^{\Lambda_1\ldots \Lambda_{k-1} \Lambda_k\setminus\{i\} \Lambda_{k+1}\ldots \Lambda_\beta}\\
  \in{}& ((a_{1i}: i\in \Lambda_1): a_{1i_0}) \cap (x_j,a_{ki}~|~j\in \Gamma, i\in \Lambda_k, 1\leq k \leq \beta)\\
  \subseteq{}&((a_{1i}: i\in \Lambda_1): x_{j^\prime}) \cap (x_j,a_{ki}~|~j\in \Gamma, i\in \Lambda_k, 1\leq k \leq \beta)\\
  ={}&(a_{1i}: i\in \Lambda_1) 
\end{align*}
where the equality holds since $\{x_j,a_{ki}:j\in \Gamma, i\in \Lambda_k, 1\leq k\leq \beta\}$ is a d-sequence in $A$, see \cite[Proposition 2.1]{huneke82}.
Now, 
\begin{align*}
b_{\Gamma}^{\Lambda_1\ldots,\Lambda_\beta} -\sum_{i\in \Lambda_1}(-1)^{|\Gamma|+\Lambda_1(i)}a_{1i}  c_{\Gamma}^{\Lambda_1\setminus\{i\} \Lambda_{2}\ldots \Lambda_\beta}&\in
(a_{1i}: i\in \Lambda_1) \cap I_{n+|\Lambda_1|+2|\Lambda_2|+\ldots+\beta |\Lambda_\beta|}\\
&=(a_{1i}: i\in \Lambda_1) I_{n+|\Lambda_1|+2|\Lambda_2|+\ldots+\beta |\Lambda_\beta|-1} 
\end{align*}
using Proposition \ref{lemma-5.1/trung}. This implies that 
$\overline{b_\Gamma^{\Lambda_1,\ldots,\Lambda_\beta}}=\xi_\Gamma^{\Lambda_1,\ldots,\Lambda_\beta}\in (a_{1i}t: i\in \Lambda_1)G(\I)\cap \h^0_\N(G(\I))=0$
since $\{a_{1i}t:i\in \Lambda_1\}$ is a d-sequence on $G(\I)$ by Proposition \ref{prop-I-equi-1}. 
Hence \[b_\Gamma^{\Lambda_1,\ldots,\Lambda_\beta}\in (a_{1i}: i\in \Lambda_1) \cap I_{n+|\Lambda_1|+2|\Lambda_2|+\ldots+\beta |\Lambda_\beta|+1}=(a_{1i}: i\in \Lambda_1)I_{n+|\Lambda_1|+2|\Lambda_2|+\ldots+\beta |\Lambda_\beta|}\] 
using Proposition \ref{lemma-5.1/trung} again. Let $g_i\in I_{n+|\Lambda_1|+2|\Lambda_2|+\ldots+\beta |\Lambda_\beta|}$ such that
$b_\Gamma^{\Lambda_1,\ldots,\Lambda_\beta}=\mathop\sum\limits_{i\in \Lambda_1}a_{1i}g_i.$ Then, for each $i\in \Lambda_1$, 
\[\eta_\Gamma^{\Lambda_1\setminus\{i\}\Lambda_2\ldots \Lambda_\beta}=\overline{c_\Gamma^{\Lambda_1\setminus\{i\}\Lambda_2\ldots \Lambda_\beta}}=\overline{c_\Gamma^{\Lambda_1\setminus\{i\}\Lambda_2\ldots \Lambda_\beta}-g_i}\]
in $G(\I)_{n+|\Lambda_1|+2|\Lambda_2|+\ldots+\beta |\Lambda_\beta|-1}.$ Therefore, for each $i\in \Lambda_1$, we may replace $c_\Gamma^{\Lambda_1\setminus\{i\}\Lambda_2\ldots \Lambda_\beta}$ by 
$c_\Gamma^{\Lambda_1\setminus\{i\}\Lambda_2\ldots \Lambda_\beta}-g_i$ for the representative of $\eta_\Gamma^{\Lambda_1\setminus\{i\}\Lambda_2\ldots \Lambda_\beta}$ in the expression \eqref{eq-main-proof-4}
of $b_\Gamma^{\Lambda_1,\ldots,\Lambda_\beta}$ which completes the proof.
\end{proof}
{\textbf{(Proof of Theorem \ref{main-theorem-last-section} continued.)}}
Let us fix  $\Gamma\subseteq [1,\nu]$, $\Lambda_1\subseteq[1,u_1],\ldots$, $\Lambda_\beta\subseteq [1,u_\beta]$ such that $|\Gamma|+|\Lambda_1|+\ldots+|\Lambda_\beta|=d.$ 
Let $|\Lambda_1|=p$. We may assume that $\Lambda_1=[1,p]$ after rearrangement of $\{a_{11},\ldots, a_{1u_1}\}$.

{\it{Claim 3}.} $b_{\Gamma^\prime}^{[1,l]\Lambda_2^\prime,\ldots \Lambda_\beta^\prime}=0$ for all $\Gamma^\prime \subseteq \Gamma$, $\Lambda_k^\prime\subseteq \Lambda_k,~2\leq k\leq \beta$ with 
$|\Gamma^\prime|+|\Lambda_2^\prime|+\ldots+|\Lambda_\beta^\prime|=d-l$ 
and for all $p\leq l\leq d$ 
after a suitable change of representations $c_{\Gamma^\prime}^{[1,l]\setminus\{i\}\Lambda_2^\prime\ldots \Lambda_\beta^\prime}$ for $i\in [1,l].$

\begin{proof}[Proof of Claim 3] We use induction on $p$. If $p=d$, then $l=d$ and $\Gamma^\prime=\Lambda_2^\prime=\ldots=\Lambda_\beta^\prime=\phi$. By Lemma \ref{Lemma-inside-main-proof}, we
have a choice of representatives $c_{\Gamma^\prime}^{[1,d]\setminus\{i\}\Lambda_2^\prime\ldots \Lambda_\beta^\prime}$ in the expression \eqref{eq-main-proof-4} such that  $b_{\Gamma^\prime}^{[1,d]\Lambda_2^\prime,\ldots \Lambda_\beta^\prime}=0$. 
Suppose $0\leq p<d$ and the assertion holds for $p+1$. Then it is enough to consider the case when $l=p$. Let  
$\Gamma^\prime \subseteq \Gamma$, $\Lambda_k^\prime\subseteq \Lambda_k,~2\leq k\leq \beta$ with $|\Gamma^\prime|+|\Lambda_2^\prime|+\ldots+|\Lambda_\beta^\prime|=d-l$ be fixed. 
Put $\Lambda_1^\prime=[1,p+1]$, then by induction hypothesis we have that 
 \[ b_{\Gamma^{\prime\prime}}^{\Lambda_1^\prime \Lambda_2^\prime\ldots \Lambda_\beta^\prime}=0 \text{ whenever } \Gamma^{\prime\prime}\subseteq \Gamma^\prime \text{ with } 
 |\Gamma^{\prime\prime}|=|\Gamma^\prime |-1\]
  and 
\[ 
b_{\Gamma^\prime}^{\Lambda_1^\prime \Lambda_2^\prime\ldots \Lambda_{k-1}^\prime \Lambda_k^{\prime\prime} \Lambda_{k+1}^\prime\ldots \Lambda_\beta^\prime}=0 \text{ whenever } \Lambda_k^{\prime\prime}\subseteq \Lambda_k^\prime \text{ with } 
|\Lambda_k^{\prime\prime}|=|\Lambda_k^\prime|-1 \text{ for } 
2\leq k\leq \beta
  \]
  up to change of representations. Then, by Lemma \ref{Lemma-inside-main-proof}, $b_{\Gamma^\prime}^{[1,l]\Lambda_2^\prime,\ldots \Lambda_\beta^\prime}=0$ after suitable 
  change of representations $c_{\Gamma^\prime}^{[1,l]\setminus\{i\}\Lambda_2^\prime\ldots \Lambda_\beta^\prime}$ for $i\in [1,l].$
This complete the proof of Claim 3.
  \end{proof}
 Now, put $l=p$, $\Gamma^\prime=\Gamma$ and $\Lambda_k^\prime=\Lambda_k$ for $2\leq k\leq \beta$ in Claim 3 to complete the proof of Theorem \ref{main-theorem-last-section}.
 \end{proof}
 
{\textbf{Application.}} In a Noetherian local ring $A$, the Hilbert
coefficients of an $\m$-primary ideal $I$ satisfies the following relation, see \cite[Theorem 2.4]{rossi-valla}, 
\begin{align}\label{eqn-app}
e_1(I)-e_1(Q)\geq 2(e_0(I)-\ell_A(A/I))-\ell_A(I/I^2+Q) 
\end{align}
where $Q\subseteq I$ is a minimal reduction of $I$. If $A$ is Cohen-Macaulay, then the above  inequality follows from the famous bounds 
of Huckaba-Marley \cite{HM94}. Corso in \cite{corso} conjectured that if $A$ is Buchsbaum and equality holds in \eqref{eqn-app} for $I=\m$, then $G(\m)$ is Buchsbaum.  
 Corso's conjecture holds more generally for $\m$-primary ideals, which is proved in \cite{ozeki12} and \cite{Ozeki13}. An alternative proof can be 
 given using Theorem \ref{theorem-intro-2}.  In fact, one can show that if $A$ is Buchsbaum and equality holds in \eqref{eqn-app} 
 then $\mathds{I}(G(I))=\mathds{I}(A)$, see \cite[Theorem 1.1]{Ozeki13}. Then, by Theorem \ref{theorem-intro-2}, we conclude that $G(I)$ is Buchsbaum.  
 Further discussions concerning Corso's conjecture is 
 subject of a subsequent paper.

\end{document}